\pdfoutput=1
\RequirePackage{ifpdf}
\ifpdf 
\documentclass[pdftex]{sigma}
\else
\documentclass{sigma}
\fi

\numberwithin{equation}{section}

\newtheorem{Theorem}{Theorem}[section]
\newtheorem{Corollary}[Theorem]{Corollary}
\newtheorem{Lemma}[Theorem]{Lemma}
\newtheorem{Proposition}[Theorem]{Proposition}
\newtheorem{Problem}[Theorem]{Problem}
 { \theoremstyle{definition}
\newtheorem{Definition}[Theorem]{Definition}
\newtheorem{Note}[Theorem]{Note}
\newtheorem{Example}[Theorem]{Example}}

\begin{document}

\newcommand{\arXivNumber}{2003.09666}

\renewcommand{\PaperNumber}{009}

\FirstPageHeading

\ShortArticleName{Double Lowering Operators on Polynomials}

\ArticleName{Double Lowering Operators on Polynomials}

\Author{Paul TERWILLIGER}

\AuthorNameForHeading{P.~Terwilliger}

\Address{Department of Mathematics, University of Wisconsin, Madison, WI 53706-1388, USA}
\Email{\href{mailto:terwilli@math.wisc.edu}{terwilli@math.wisc.edu}}

\ArticleDates{Received September 15, 2020, in final form January 19, 2021; Published online January 28, 2021}

\Abstract{Recently Sarah Bockting-Conrad introduced the double lowering operator $\psi$ for a tridiagonal pair. Motivated by $\psi$ we consider the following problem about polynomials. Let $\mathbb F$ denote an algebraically closed field. Let $x$ denote an indeterminate, and let $\mathbb F\lbrack x \rbrack$ denote the algebra consisting of the polynomials in $x$ that have all coefficients in $\mathbb F$. Let $N$ denote a positive integer or $\infty$. Let $\lbrace a_i\rbrace_{i=0}^{N-1}$, $\lbrace b_i\rbrace_{i=0}^{N-1}$ denote scalars in $\mathbb F$ such that $\sum_{h=0}^{i-1} a_h \not= \sum_{h=0}^{i-1} b_h$ for $1 \leq i \leq N$. For $0 \leq i \leq N$ define polynomials $\tau_i, \eta_i \in \mathbb F\lbrack x \rbrack$ by $\tau_i = \prod_{h=0}^{i-1} (x-a_h)$ and $\eta_i = \prod_{h=0}^{i-1} (x-b_h)$. Let $V$ denote the subspace of $\mathbb F\lbrack x \rbrack$ spanned by $\lbrace x^i\rbrace_{i=0}^N$. An element $\psi \in \operatorname{End}(V)$ is called double lowering whenever $\psi \tau_i \in \mathbb F \tau_{i-1}$ and $\psi \eta_i \in \mathbb F \eta_{i-1}$ for $0 \leq i \leq N$, where $\tau_{-1}=0$ and $\eta_{-1}=0$. We give necessary and sufficient conditions on $\lbrace a_i\rbrace_{i=0}^{N-1}$, $\lbrace b_i\rbrace_{i=0}^{N-1}$ for there to exist a nonzero double lowering map. There are four families of solutions, which we describe in detail.}


\Classification{33D15; 15A21}

\section{Introduction}\label{section1}

This paper is mainly about polynomials and special functions,
but in order to motivate things we first discuss a topic in linear
algebra. The topic has to do with tridiagonal pairs~\cite{TD00}
and their associated double lowering operator
\cite{bockting1,bockting,bocktingQexp, bocktingTer}.
A~reader unfamiliar with tridiagonal pairs can safely skip to Section~\ref{section2}.
Let $V$ denote a vector space with finite positive dimension.
A~tridiagonal pair on~$V$
is an ordered pair of linear maps
$A\colon V\to V$
and $A^*\colon V\to V$ that satisfy the following conditions:
\begin{enumerate}\itemsep=0pt
\item[\rm (i)] each of $A$, $A^*$ is diagonalizable;
\item[\rm (ii)] there exists an ordering $\lbrace V_i\rbrace_{i=0}^d$ of the eigenspaces
of $A$ such that
\begin{gather*}
A^* V_i \subseteq V_{i-1} + V_i + V_{i+1}, \qquad 0 \leq i \leq d,
\end{gather*}
where $V_{-1}=0$ and $V_{d+1}=0$;
\item[\rm (iii)] there exists an ordering $\lbrace V^*_i\rbrace_{i=0}^\delta$
of the eigenspaces of $A^*$ such that
\begin{gather*}
A V^*_i \subseteq V^*_{i-1} + V^*_i + V^*_{i+1}, \qquad 0 \leq i \leq \delta,
\end{gather*}
where $V^*_{-1}=0$ and $V^*_{\delta+1}=0$;
\item[\rm (iv)] there does not exist a subspace $W \subseteq V$
such that
$AW \subseteq W$ and
$A^*W \subseteq W$ and
$W \not=0$ and
$W \not=V$.
\end{enumerate}
The tridiagonal pair concept originated
in the theory of $Q$-polynomial distance-regular graphs~\cite{bannaiIto,BCN} where it
is used to describe how the adjacency matrix is related to each
dual adjacency matrix
\cite[Example~1.4]{TD00},
\cite[Section~3]{tersub1}.
Since that origin,
the tridiagonal pair concept has found applications to
all sorts of topics in special functions (orthogonal polynomials of the
Askey-scheme
\cite{aw, koekoek, NomTerK, ter2004, ter2005},
the Askey--Wilson algebra
\cite{zhedmut,
univTer,
vidter, zhedhidden}),
Lie theory
(the ${\mathfrak {sl}}_2$ loop algebra
\cite{IT:Krawt},
 the tetrahedron algebra \cite{Ha,Ev}),
statistical mechanics (the Onsager algebra \cite{DateRoan2, HT} and $q$-Onsager algebra
 \cite{bas2, bas5,
 augmented,
 qSerre,
pospart,
qOnsUniv}),
and quantum groups (the equitable presentation
\cite{
alnajjar,
equit, uawe},
 the quantum affine
${\mathfrak {sl}}_2$ algebra
\cite{alnajjarC,tdanduq,
nonnil,tdqrac},
 $L$-operators
\cite{miki,terLop}).
For more information about the above topics, see
\cite{NomTerTB, madrid}
and the references therein.

Let $A$, $A^*$ denote a tridiagonal pair on $V$, as in the above definition.
By
\cite[Lemma~4.5]{TD00}
the integers $d$ and $\delta$ from (ii), (iii) are equal;
this common value is called the diameter of the pair.
For $0 \leq i \leq d$ let $\theta_i$ (resp.~$\theta^*_i$) denote
the eigenvalue of $A$ (resp.~$A^*$) for the eigenspace $V_i$
(resp.~$V^*_i$). By \cite[Theorem~11.1]{TD00} the scalars
\begin{gather*}
\frac{\theta_{i-2}-\theta_{i+1}}{\theta_{i-1}-\theta_i},
\qquad
\frac{\theta^*_{i-2}-\theta^*_{i+1}}{\theta^*_{i-1}-\theta^*_i}
\end{gather*}
are equal and independent of $i$ for $2 \leq i \leq d-1$. For
this constraint the solutions can be given in closed form
\cite[Theorem~11.2]{TD00}.
The ``most general'' solution is called $q$-Racah, and will
be described shortly.

By construction the vector space $V$
has a direct sum decomposition into the eigenspaces
$\lbrace V_i\rbrace_{i=0}^d$
of $A$
and the eigenspaces $\lbrace V^*_i\rbrace_{i=0}^d$ of $A^*$.
The vector space $V$ has
two other direct sum decompositions of interest,
called the first split decomposition
$\lbrace U_i\rbrace_{i=0}^d$ and
second split decomposition
$\big\lbrace U^\Downarrow_i\big\rbrace_{i=0}^d$.
By \cite[Theorem~4.6]{TD00}
the first split decomposition satisfies
\begin{gather*}
U_0 + U_1+ \cdots + U_i = V^*_0 + V^*_1 + \cdots + V^*_i,
\\
U_i + U_{i+1}+ \cdots + U_d = V_i + V_{i+1} + \cdots + V_d
\end{gather*}
for $0 \leq i \leq d$.
By \cite[Theorem~4.6]{TD00}
the second split decomposition satisfies
\begin{gather*}
U^\Downarrow_0 + U^\Downarrow_1+ \cdots + U^\Downarrow_i = V^*_0 + V^*_1 + \cdots + V^*_i,
\\
U^\Downarrow_i + U^\Downarrow_{i+1}+ \cdots + U^\Downarrow_d =
V_0 + V_{1} + \cdots + V_{d-i}
\end{gather*}
for $0 \leq i \leq d$.
By \cite[Theorem~4.6]{TD00},
\begin{alignat*}{3}
& (A-\theta_i I)U_i \subseteq U_{i+1},
\qquad && (A^*-\theta^*_i I) U_i \subseteq U_{i-1},&
\\
& (A-\theta_{d-i} I)U^\Downarrow_i \subseteq U^\Downarrow_{i+1},
 \qquad&& (A^*-\theta^*_i I) U^\Downarrow_i \subseteq U^\Downarrow_{i-1}&
\end{alignat*}
for $0 \leq i \leq d$, where $U_{-1}=0$, $U_{d+1}=0$
and $U^\Downarrow_{-1}=0$, $U^\Downarrow_{d+1}=0$.

In \cite[Sections~11 and~15]{bockting1} Sarah Bockting-Conrad
introduces a linear map
$\Psi \colon V\to V$ such that
\begin{gather*}
\Psi U_i \subseteq U_{i-1}, \qquad
\Psi U^\Downarrow_i \subseteq U^\Downarrow_{i-1},
\qquad 0 \leq i \leq d.
\end{gather*}
This map is called the double lowering operator or Bockting operator.
In
\cite[Sections~9 and~15]{bockting1}
Bockting-Conrad introduces an invertible linear map
$\Delta\colon V\to V$ that commutes with $\Psi$ and
 sends $U_i$ onto $U^\Downarrow_i$ for $0 \leq i \leq d$.
The maps $\Psi$ and $\Delta$
are related in the following way.
For $0 \leq i \leq d$ define two polynomials
\begin{gather}
\tau_i = (x-\theta_0)(x-\theta_1)\cdots (x-\theta_{i-1}),
\label{eq:Tau}\\
\eta_i = (x-\theta_d)(x-\theta_{d-1})\cdots (x-\theta_{d-i+1})
\label{eq:Eta}
\end{gather}
in a variable $x$.
Define the scalars
\begin{gather*}
\vartheta_i = \sum_{h=0}^{i-1} \frac{\theta_h - \theta_{d-h}}{\theta_0 - \theta_d}, \qquad 1 \leq i \leq d.
\end{gather*}
By \cite[Theorem~17.1]{bockting1},
\begin{gather}
\Delta = \sum_{i=0}^d \frac{\eta_i (\theta_0)}{\vartheta_1 \vartheta_2
\cdots \vartheta_i} \Psi^i,
\qquad
\Delta^{-1} = \sum_{i=0}^d \frac{\tau_i (\theta_d)}{\vartheta_1 \vartheta_2
\cdots \vartheta_i} \Psi^i
\label{eq:DeltaDeltaInv}
\end{gather}
provided that each of $\vartheta_1, \vartheta_2,\ldots,\vartheta_d$ is nonzero.

Shortly we will describe $\Psi$ and $\Delta$ in more detail,
but first we restrict to the $q$-Racah case.
In this case there exist nonzero scalars $a$, $b$, $q$ such that $q^4 \not=1$ and
\begin{gather*}
\theta_i = a q^{d-2i}+ a^{-1} q^{2i-d}, \qquad
\theta^*_i = b q^{d-2i}+ b^{-1} q^{2i-d}
\end{gather*}
for $0 \leq i \leq d$.
Define a linear map $K\colon V\to V$
such that for $0 \leq i \leq d$, $U_i$ is an eigenspace of~$K$ with eigenvalue $q^{d-2i}$.
Define a linear map $B\colon V\to V$
such that for $0 \leq i \leq d$, $U^\Downarrow_i$ is an eigenspace of~$B$ with eigenvalue~$q^{d-2i}$. For notational convenience
define $\psi = \big(q-q^{-1}\big)\big(q^d-q^{-d}\big)\Psi$.
By \cite[Lemma~5.3]{bocktingQexp}
and \cite[Lemma~5.4]{bockting},
\begin{gather*}
B\Delta=\Delta K,
\qquad
K \psi = q^2 \psi K,
\qquad
B \psi = q^2 \psi B.
\end{gather*}
By \cite[Theorem~9.8]{bockting} the map $\psi$ is equal to
each of the following:
\begin{alignat*}{3}
& \frac{I - BK^{-1}}{q\big(aI - a^{-1}BK^{-1}\big)}, \qquad&&
\frac{I - KB^{-1}}{q\big(a^{-1}I - aKB^{-1}\big)},&\\
& \frac{q\big(I - K^{-1}B\big)}{aI - a^{-1}K^{-1}B}, \qquad&&
\frac{q\big(I - B^{-1}K\big)}{a^{-1}I - aB^{-1}K}.&
\end{alignat*}
This result is used in
\cite[Theorem~9.9]{bockting}
to obtain
\begin{gather*}
a K^2 -
\frac{a^{-1}q-aq^{-1}}{q-q^{-1}}KB
-
\frac{aq-a^{-1}q^{-1}}{q-q^{-1}}BK
+ a^{-1}B^2=0.
\end{gather*}
By
(\ref{eq:DeltaDeltaInv}) and
\cite[Theorem~7.2]{bocktingQexp},
\begin{gather*}
\Delta = \exp_q \left( \frac{a}{q-q^{-1}}\psi \right)
\exp_{q^{-1}} \left( - \frac{a^{-1}}{q-q^{-1}}\psi \right).
\end{gather*}
Motivated by this factorization, in \cite[Sections~6,7]{bocktingQexp}
Bockting-Conrad introduces an invertible linear map $M\colon V\to V$ such that
\begin{gather*}
K \exp_q \left( \frac{a^{-1}}{q-q^{-1}}\psi \right)
=\exp_{q} \left( \frac{a^{-1}}{q-q^{-1}}\psi \right) M,
\\
B \exp_q \left( \frac{a}{q-q^{-1}}\psi \right)
=
\exp_{q} \left( \frac{a}{q-q^{-1}}\psi \right) M.
\end{gather*}
By \cite[Section~6]{bocktingQexp},
\begin{gather*}
M = \frac{a K-a^{-1}B}{a-a^{-1}},
 \qquad
M \psi = q^2 \psi M.
\end{gather*}
By \cite[Lemma~6.2]{bocktingQexp},
 $M$ is equal to each of
\begin{alignat*}{3}
&\big(I-a^{-1}q\psi\big)^{-1}K, \qquad&& K\big(I-a^{-1}q^{-1}\psi\big)^{-1},&
\\ & (I-aq\psi)^{-1}B, \qquad&&
B\big(I-aq^{-1}\psi\big)^{-1}.&
\end{alignat*}
By \cite[Lemma~6.7]{bocktingQexp},
\begin{gather*}
\frac{q M^{-1}K-q^{-1}K M^{-1}}{q-q^{-1}} = I,
\qquad
\frac{q M^{-1}B-q^{-1}B M^{-1}}{q-q^{-1}} = I.
\end{gather*}

We just listed many results about
$\psi$, $\Delta$, $K$, $B$, $M$. In the present paper, we interpret
these results in terms of polynomials. The polynomials in
question are essentially
(\ref{eq:Tau}),
(\ref{eq:Eta})
although we adopt a more general point of view.
In the next section we will describe a problem about polynomials,
and for the rest of the paper we will describe the solution.
In this description we will encounter
analogs of the above results.
We hope that the above results are illuminated
by our description.

\section{Definitions and first steps}\label{section2}

We now begin our formal argument.
The following assumptions and notational conventions
apply throughout the paper.
Recall the natural numbers $\mathbb N = \lbrace 0,1,2,\ldots\rbrace$
and integers $\mathbb Z = \lbrace 0, \pm 1,\pm 2,\ldots \rbrace$.
Let $\mathbb F$ denote an algebraically closed field.
Every vector space
discussed in this paper is over $\mathbb F$.
Every algebra discussed in this paper is associative,
over $\mathbb F$, and has a~multiplicative identity.
Let~$x$ denote an indeterminate. Let $\mathbb F\lbrack x
\rbrack$ denote the algebra consisting of the polynomials in
$x $ that have all coefficients in $\mathbb F$.
Throughout the paper we use the following convention:
{\it the symbols $N$, $n$ refer to an integer or $\infty$; the
symbols $i$, $j$, $k$ refer to an integer.}

We now describe a problem about polynomials.
Let $N$ denote a positive integer or $\infty$.
We consider an ordered pair of sequences
\begin{gather}
\lbrace a_i \rbrace_{i=0}^{N-1}, \qquad
\lbrace b_i \rbrace_{i=0}^{N-1}
\label{eq:data}
\end{gather}
such that $a_i, b_i \in \mathbb F$ for $0 \leq i \leq N-1$.
To avoid degenerate situations, we assume that
\begin{gather}
a_0 + a_1 + \cdots + a_{i-1} \not=
b_0 + b_1 + \cdots + b_{i-1},
\qquad 1 \leq i \leq N.
\label{eq:nondeg}
\end{gather}
We call the ordered pair~(\ref{eq:data})
the {\it data}.
For $0 \leq i \leq N$ define polynomials
$\tau_i$, $\eta_i \in
\mathbb F\lbrack x \rbrack$ by
\begin{gather}
\tau_i = (x - a_0)(x-a_1)\cdots (x - a_{i-1}),\label{eq:tau}\\
\eta_i = (x - b_0)(x-b_1)\cdots (x - b_{i-1}).\label{eq:eta}
\end{gather}
The polynomials $\tau_i$, $\eta_i$ are monic of degree $i$.
For notational convenience define $\tau_{-1} = 0$ and
$\eta_{-1} = 0$.
Let $V$ denote the subspace of
$\mathbb F\lbrack x \rbrack$
spanned by $\lbrace x^i \rbrace_{i=0}^{N}$.
Each of
$\lbrace \tau_i \rbrace_{i=0}^{N}$,
$\lbrace \eta_i \rbrace_{i=0}^{N}$ is a~basis for $V$.
Let $\operatorname{End}(V)$ denote the algebra consisting of the
$\mathbb F$-linear maps from $V$ to $V$.
Define $\Delta \in \operatorname{End}(V)$ such that
\begin{gather}
\Delta \tau_i = \eta_i, \qquad 0 \leq i \leq N.
\label{def:Delta}
\end{gather}
Note that $\Delta$ is invertible.

\begin{Definition}\label{def:dlowering}
An element $\psi \in \operatorname{End}(V)$ is called {\it double lowering}
(with respect to the given data)
whenever both
\begin{gather*}
\psi \tau_i \in \mathbb F \tau_{i-1}, \qquad
\psi \eta_i \in \mathbb F \eta_{i-1}
\end{gather*}
for $0 \leq i \leq N$.
\end{Definition}

\begin{Definition}\label{def:dl}
Define
\begin{gather*}
\mathcal L = \lbrace \psi \in \operatorname{End}(V) \,|\, \psi \;
{\mbox{\rm is double lowering}}
\rbrace.
\end{gather*}
Note that
$\mathcal L $ is a subspace of the vector space
$\operatorname{End}(V)$. We call
$\mathcal L $ the {\it double lowering space} for
the given data.
\end{Definition}

\begin{Definition}\label{def:feasible} The data
(\ref{eq:data})
is called {\it double lowering} whenever the double lowering space
$\mathcal L \not=0$.
\end{Definition}

\begin{Problem} Find necessary and sufficient conditions
for the data
\eqref{eq:data}
to be double lowering. In this case describe
$\mathcal L$ and $\Delta$.
\end{Problem}

The above problem is solved in the present paper.
The necessary and sufficient conditions are given in
Theorem~\ref{thm:hr}. By that theorem, there are four cases.
For the first three cases, $\mathcal L$ and $\Delta$ are described
in Section~\ref{section6}. For the fourth case, $\mathcal L$ and $\Delta$
are described
in Section~\ref{section13}. We would like to acknowledge that the above
problem was previously solved by R.~Vidunas
under the assumption
that $a_i =\theta_i$ and $b_i=\theta_{N-i}$ for $0 \leq i \leq N-1$,
with $\lbrace \theta_i\rbrace_{i=0}^N$ mutually distinct~\cite{vidunas}.

We have some remarks. The polynomials~(\ref{eq:tau}), (\ref{eq:eta}) satisfy
\begin{gather}\label{eq:example}
\tau_0=1, \qquad
\tau_1=x-a_0, \qquad
\eta_0=1, \qquad
\eta_1=x-b_0.
\end{gather}
Moreover
\begin{gather}
 (x-a_i)\tau_i = \tau_{i+1}, \qquad
 (x-b_i)\eta_i = \eta_{i+1}, \qquad 0 \leq i \leq N-1.\label{eq:xrec}
\end{gather}

\begin{Lemma}\label{lem:dlcomment}Assume that $\psi \in \operatorname{End}(V)$ is double lowering. Then
$\psi 1 =0$. Moreover
\begin{gather}\label{eq:dlcomment}
\psi \tau_1 = \psi x = \psi \eta_1,
\end{gather}
and this common value is contained in $\mathbb F$.
\end{Lemma}
\begin{proof} Apply $\psi$ to each side of the equations in~(\ref{eq:example}), and use Definition~\ref{def:dlowering}.
\end{proof}

For $0 \leq n \leq N$ define
\begin{gather}
V_n = \operatorname{Span}\big\lbrace x^i \big\rbrace_{i=0}^{n}.\label{eq:Vn}
\end{gather}
We have $V_N=V$.
 We have $\dim (V_i)= i+1$ for $0 \leq i \leq N$,
and $V_{i-1} \subseteq V_i$ for $1 \leq i \leq N$.
For notational convenience define $V_{-1} = 0$.
\begin{Lemma}\label{lem:Vnbasis}
For $0 \leq n \leq N$,
each of
\begin{gather*}
\lbrace \tau_i \rbrace_{i=0}^{n},
\qquad
\lbrace \eta_i \rbrace_{i=0}^{n}
\end{gather*}
is a basis for $V_n$.
\end{Lemma}
\begin{proof} Each of $\tau_i$, $\eta_i$ has degree $i$ for $0 \leq i \leq N$.
\end{proof}

\begin{Lemma}\label{lem:psiVn}
For $\psi \in \mathcal L$ and $0 \leq i \leq N$,
\begin{gather*}
\psi V_i \subseteq V_{i-1}.
\end{gather*}
Moreover $\psi^{i+1}V_i =0$.
\end{Lemma}
\begin{proof} By Definition
\ref{def:dlowering}
and Lemma \ref{lem:Vnbasis}.
\end{proof}

For $T \in \operatorname{End}(V)$,
$T$ is called {\it nilpotent} whenever there exists a positive
integer $j$ such that $T^j=0$.
The map $T$ is called {\it locally nilpotent} whenever
for all $v \in V$ there exists a positive integer $j$ such
that $T^jv=0$. If $T$ is nilpotent then $T$ is locally nilpotent.
For $N\not=\infty$, if $T$ is locally nilpotent then $T$ is nilpotent.

\begin{Lemma}\label{lem:nilp}
Each element of $\mathcal L$ is locally nilpotent.
If $N\not=\infty$ then each element of $\mathcal L$ is nilpotent.
\end{Lemma}
\begin{proof} By Lemma \ref{lem:psiVn} and the comments below it.
\end{proof}

\begin{Lemma}\label{lem:lterm}
For $1 \leq i \leq N$,
\begin{enumerate}\itemsep=0pt
\item[$(i)$] in $\tau_i$ the coefficient of $x^{i-1}$ is
\begin{gather*}
- a_0 - a_1 - \cdots - a_{i-1};
\end{gather*}
\item[$(ii)$] in $\eta_i$ the coefficient of $x^{i-1}$ is
\begin{gather*}
- b_0 - b_1 - \cdots - b_{i-1}.
\end{gather*}
\end{enumerate}
\end{Lemma}
\begin{proof} Use (\ref{eq:tau}), (\ref{eq:eta}).
\end{proof}

\begin{Lemma}\label{lem:etaMtau}
For $1 \leq i \leq N$,
\begin{enumerate}\itemsep=0pt
\item[$(i)$] $\eta_i - \tau_i \in V_{i-1}$;
\item[$(ii)$] in $\eta_i - \tau_i$ the coefficient of
$x^{i-1} $ is
\begin{gather*}
a_0 + a_1 + \cdots + a_{i-1} -
b_0 - b_1 - \cdots - b_{i-1}.
\end{gather*}
\end{enumerate}
\end{Lemma}
\begin{proof} (i) Each of $\tau_i$, $\eta_i$ is monic with degree $i$.
(ii) Use Lemma \ref{lem:lterm}.
\end{proof}

\begin{Definition}\label{def:A}
Define a map $A\colon V_{N-1} \to V_N$, $v \mapsto x v$.
Note that $A$ is $\mathbb F$-linear.
\end{Definition}

\begin{Lemma} We have
\begin{gather*}
A V_{n-1} \subseteq V_n, \qquad 1 \leq n \leq N.
\end{gather*}
\end{Lemma}
\begin{proof} By
(\ref{eq:Vn}).
\end{proof}
\begin{Lemma}\label{lem:Ataui}
For $0 \leq i \leq N-1$,
\begin{gather*}
A \tau_i = a_i \tau_i + \tau_{i+1},
\qquad
A \eta_i = b_i \eta_i + \eta_{i+1}.
\end{gather*}
\end{Lemma}
\begin{proof} By
(\ref{eq:xrec})
and
Definition~\ref{def:A}.
\end{proof}

We mention an elementary result for later use.
\begin{Lemma}\label{lem:gs}
Assume that $T \in \operatorname{End}(V)$
is locally nilpotent. Then $I-T$ is invertible, and $(I-T)^{-1}= \sum_{i=0}^N T^i$.
\end{Lemma}

\section{Adjusting the data}\label{section3}

In this section we describe how the double lowering space is affected when we adjust the data in an affine way.

Let ${\rm GL}_2(\mathbb F)$ denote the group of
invertible $2\times 2$ matrices that have all entries in $\mathbb F$.
\begin{Definition}\label{def:G}
Let $G$ denote the subgroup of
${\rm GL}_2(\mathbb F)$ consisting of the matrices
\begin{gather*}
	 \left(
	 \begin{matrix}
	 1 & t \\
	 0 & s
		 \end{matrix}
		 \right),
\qquad 0 \not=s\in \mathbb F, \qquad t \in \mathbb F.
\end{gather*}
The above matrix is denoted $g(s,t)$.
\end{Definition}

\begin{Lemma}\label{lem:group}
With reference to Definition~{\rm \ref{def:G}},
\begin{enumerate}\itemsep=0pt
\item[$(i)$]
$g(s,t) g(S,T) = g(sS,T+tS)$;
\item[$(ii)$]
the inverse of $g(s,t)$ is $g\big(s^{-1}, -s^{-1}t\big)$.
\end{enumerate}
\end{Lemma}
\begin{proof} Routine matrix multiplication.
\end{proof}

For an algebra $\mathcal A$, an {\it automorphism of
$\mathcal A$} is an algebra isomorphism $\mathcal A\to \mathcal A$.

\begin{Lemma} The group $G$ acts on the algebra
${\mathbb F}\lbrack x \rbrack$ as a group of automorphisms
in the following way: each element $g(s,t) \in G$ sends
$x \mapsto s x + t$.
\end{Lemma}
\begin{proof} This is routinely checked using
Lemma~\ref{lem:group}.
\end{proof}

Recall
from Definition~\ref{def:dl}
the double lowering space
$\mathcal L$ for the data
(\ref{eq:data}).
Pick $0 \not=s\in \mathbb F$ and $t\in \mathbb F$.
Let $\mathcal L'$ denote the double lowering space for the data
\begin{gather*}
\lbrace sa_i + t\rbrace_{i=0}^{N-1}, \qquad
\lbrace sb_i + t\rbrace_{i=0}^{N-1}.
\end{gather*}

\begin{Proposition}\label{prop:LLp}
The following $(i)$--$(iii)$ hold for the above scalars $s$, $t$ and $g=g(s,t)$:
\begin{enumerate}\itemsep=0pt
\item[$(i)$]
there exists an $\mathbb F$-linear map
$\mathcal L \to \mathcal L'$,
 $\psi \mapsto g^{-1} \psi g$;
\item[$(ii)$]
there exists an $\mathbb F$-linear map
$\mathcal L' \to \mathcal L$,
$\zeta \mapsto g \zeta g^{-1}$;
\item[$(iii)$] the maps in $(i)$, $(ii)$ above are inverses,
and hence bijections.
\end{enumerate}
\end{Proposition}
\begin{proof} (i) For $\psi \in \mathcal L$ we show that
$g^{-1}\psi g \in \mathcal L'$.
For $\alpha \in \mathbb F$
define $\alpha' = s \alpha + t$. For $0 \leq i \leq N$
define
\begin{gather*}
\tau'_i = (x - a'_0)(x-a'_1)\cdots (x - a'_{i-1}),\\
\eta'_i = (x - b'_0)(x-b'_1)\cdots (x - b'_{i-1}).
\end{gather*}
For notational convenience define $\tau'_{-1}=0$ and
$\eta'_{-1}=0$.
Pick an integer $i$, $0 \leq i \leq N$.
One checks that $g$ sends
\begin{gather*}
\tau'_i \mapsto s^i \tau_i, \qquad
\eta'_i \mapsto s^i \eta_i.
\end{gather*}
By Definition~\ref{def:dlowering}, $\psi$ sends
$\tau_i$ (resp. $\eta_i$) to a scalar multiple of
$\tau_{i-1}$ (resp.~$\eta_{i-1}$).
By these comments $g^{-1} \psi g$ sends
$\tau'_i$ (resp.~$\eta'_i$) to a scalar multiple of
$\tau'_{i-1}$ (resp.~$\eta'_{i-1}$).
Therefore \mbox{$g^{-1} \psi g \in \mathcal L'$}.
We have shown that there exists a map
 $\mathcal L \to \mathcal L'$,
$\psi \mapsto g^{-1} \psi g$.
By construction this map is $\mathbb F$-linear.

(ii) Similar to the proof of (i) above. (iii) By construction.
\end{proof}

\begin{Corollary}\label{cor:affine}
Pick $0 \not=s \in \mathbb F$ and
$t \in \mathbb F$. Then the data~\eqref{eq:data} is double lowering if and only if
the data
\begin{gather*}
\lbrace sa_i + t\rbrace_{i=0}^{N-1}, \qquad
\lbrace sb_i + t\rbrace_{i=0}^{N-1}
\end{gather*}
is double lowering.
\end{Corollary}
\begin{proof} By Definition~\ref{def:feasible}
and Proposition~\ref{prop:LLp}.
\end{proof}

We have a comment.

\begin{Lemma}\label{lem:a0b0free}
Referring to the data~\eqref{eq:data},
for distinct $a'_0, b'_0 \in \mathbb F$ define
\begin{gather*}
s = \frac{a'_0 - b'_0}{ a_0-b_0}, \qquad
t= \frac{a_0 b'_0 - a'_0 b_0}{ a_0-b_0}.
\end{gather*}
Then $s\not=0$ and
\begin{gather*}
 s a_0 + t = a'_0, \qquad
 s b_0 + t = b'_0.
\end{gather*}
\end{Lemma}
\begin{proof} Routine.
\end{proof}

Corollary \ref{cor:affine} and Lemma~\ref{lem:a0b0free} show that for double lowering data~(\ref{eq:data}),
the scalars $a_0$ and $b_0$ are ``free'', with $a_0\not=b_0$ being
the only constraint.

\section[The parameters vartheta i]{The parameters $\boldsymbol{\vartheta_i}$}\label{section4}

We continue to discuss the double lowering space $\mathcal L$ for the data~(\ref{eq:data}). In this section we use $\mathcal L$
 to define some scalars $\lbrace \vartheta_i \rbrace_{i=0}^N$ that will play a role in our theory.

\begin{Definition}\label{def:vartheta}
For $0 \leq i \leq N$ define
\begin{gather*}
\vartheta_i = \frac{a_0 + a_1 + \cdots + a_{i-1}-b_0 - b_1 - \cdots - b_{i-1}}
{a_0 -b_0}.
\end{gather*}
\end{Definition}
Referring to Definition~\ref{def:vartheta}, we have $\vartheta_0 = 0$ and $\vartheta_1 = 1$.
By~(\ref{eq:nondeg}) we have
\begin{gather*}
\vartheta_i \not=0, \qquad 1 \leq i \leq N.
\end{gather*}

\begin{Proposition}\label{prop:psiVT}
The following $(i)$--$(iii)$ hold
for $\psi \in \mathcal L$ and $1 \leq i \leq N$:
\begin{enumerate}\itemsep=0pt
\item[$(i)$] for the polynomial $\psi x^i$ the coefficient
of $x^{i-1}$ is $\vartheta_i \psi x$;
\item[$(ii)$] $\psi \tau_i = (\vartheta_i \psi x) \tau_{i-1}$;
\item[$(iii)$] $\psi \eta_i = (\vartheta_i \psi x) \eta_{i-1}$.
\end{enumerate}
\end{Proposition}
\begin{proof} We use induction on $i$. First assume that $i=1$,
and recall $\vartheta_1=1$. Assertion (i) is vacuously true.
Assertions (ii), (iii) hold by
(\ref{eq:dlcomment}) and since $\tau_0=1=\eta_0$.
 Next assume that $i\geq 2$. By Lemma~\ref{lem:psiVn},
 $\psi x^i
\in V_{i-1}$. Let the scalar $\alpha_i$ be the
coefficient of $x^{i-1}$ in $\psi x^i$.
Since $\psi \in \mathcal L$,
$\psi \tau_i \in \mathbb F \tau_{i-1}$. Since each of
$\tau_i$, $\tau_{i-1}$ is monic,
\begin{gather}
\psi \tau_i= \alpha_i \tau_{i-1}.
\label{eq:one}
\end{gather}
Similarly,
\begin{gather}
\psi \eta_i= \alpha_i \eta_{i-1}.
\label{eq:two}
\end{gather}
We show that $\alpha_i = \vartheta_i \psi x$.
Using~(\ref{eq:one}),
(\ref{eq:two}) we see that
\begin{gather}
\label{eq:three}
(a_0 - b_0) \vartheta_{i-1} x^{i-2}
\end{gather}
times
\begin{gather}
\alpha_i - \vartheta_i \psi x
\label{eq:four}
\end{gather}
is equal to
\begin{gather}
\psi(\eta_i - \tau_i) - (a_0 - b_0) \vartheta_i \psi x^{i-1}
\label{eq:five}
\end{gather}
minus $\alpha_i$ times
\begin{gather}
\eta_{i-1} - \tau_{i-1} - (a_0 - b_0) \vartheta_{i-1} x^{i-2}
\label{eq:six}
\end{gather}
plus $(a_0-b_0) \vartheta_i$ times
\begin{gather}
\psi x^{i-1}-(\psi x) \vartheta_{i-1} x^{i-2}.\label{eq:seven}
\end{gather}
By Lemma~\ref{lem:etaMtau} and Definition~\ref{def:vartheta},
\begin{gather}
\eta_i - \tau_i - (a_0 - b_0)\vartheta_i x^{i-1} \in V_{i-2}.
\label{eq:eight}
\end{gather}
In this inclusion, we apply $\psi$ to each side and use
Lemma~\ref{lem:psiVn}
to find that~(\ref{eq:five}) is contained in~$V_{i-3}$.
In~(\ref{eq:eight}) we replace
$i$ by $i-1$, to find that~(\ref{eq:six}) is contained in~$V_{i-3}$.
By induction (\ref{eq:seven}) is contained in~$V_{i-3}$.
By these comments the polynomial~(\ref{eq:three}) times the scalar~(\ref{eq:four}) is
contained in~$V_{i-3}$.
Consider the factors in the polynomial~(\ref{eq:three}).
We have
$a_0 - b_0 \not=0$ and
$\vartheta_{i-1} \not=0$ and
$x^{i-2} \not\in V_{i-3}$.
So the polynomial~(\ref{eq:three})
is not contained in $V_{i-3}$. Consequently the scalar~(\ref{eq:four}) is zero, so
$\alpha_i = \vartheta_i \psi x$.
The result follows from this
and~(\ref{eq:one}),~(\ref{eq:two}).
\end{proof}

\begin{Corollary}\label{cor:psiL}
The map $\mathcal L \to \mathbb F$, $\psi \mapsto \psi x $ is injective.
\end{Corollary}
\begin{proof}
For $\psi \in \mathcal L$ such that $\psi x=0$,
we show that $\psi=0$.
The vector space $V$ has a basis~$\lbrace \tau_i \rbrace_{i=0}^N$.
By Definition~\ref{def:dlowering},
$\psi \tau_0=0$.
By Proposition~\ref{prop:psiVT}(ii), $\psi \tau_i=0$ for $1 \leq i \leq N$.
By these comments $\psi=0$.
\end{proof}

\begin{Corollary}\label{cor:L01}
Assume that $\mathcal L\not=0$. Then
the map in Corollary~{\rm \ref{cor:psiL}} is a bijection.
 Moreover $\mathcal L$ has dimension $1$.
\end{Corollary}
\begin{proof} By Corollary~\ref{cor:psiL}.
\end{proof}

\begin{Definition}\label{def:norm}
An element $\psi \in \mathcal L$ is called {\it normalized}
whenever $\psi x=1$.
\end{Definition}

\begin{Lemma}\label{lem:norm}
The double lowering space
$\mathcal L$ contains a normalized element if and only
if
$\mathcal L \not=0$; in this case
the normalized element is unique.
\end{Lemma}
\begin{proof} By Corollary~\ref{cor:L01}
and Definition~\ref{def:norm}.
\end{proof}

\begin{Lemma}\label{lem:normBasis}
Assume that $\mathcal L\not=0$.
Then the vector space $\mathcal L$ has a basis consisting
of its normalized element.
\end{Lemma}
\begin{proof} The vector space $\mathcal L$ has dimension~1, and
its normalized element is nonzero.
\end{proof}

\begin{Lemma}\label{lem:charNorm}
For $\psi \in \operatorname{End}(V)$ the following
are equivalent:
\begin{enumerate}\itemsep=0pt
\item[$(i)$] $\psi \in \mathcal L$ and $\psi$ is normalized;
\item[$(ii)$] for $0 \leq i \leq N$ both
\begin{gather}
\label{eq:Normal}
\psi \tau_i = \vartheta_i \tau_{i-1}, \qquad \psi \eta_i = \vartheta_i \eta_{i-1}.
\end{gather}
\end{enumerate}
\end{Lemma}
\begin{proof}
(i) $\Rightarrow$ (ii) Set $\psi x=1$
in Proposition~\ref{prop:psiVT}(ii), (iii).
 (ii) $\Rightarrow$ (i) We have $\psi \in \mathcal L$ by
Definition~\ref{def:dl} and~(\ref{eq:Normal}). To see that $\psi$ is normalized, set
$i=1$ in~(\ref{eq:Normal}) and use Lemma~\ref{lem:dlcomment} to obtain
$\psi x =1$.
\end{proof}

\section[Describing L using Delta]{Describing $\boldsymbol{\mathcal L}$ using $\boldsymbol{\Delta}$}\label{section5}

We continue to discuss the double lowering space
$\mathcal L$ for the data~(\ref{eq:data}). In this section we describe~$\mathcal L$ using the map~$\Delta$ from~(\ref{def:Delta}).

\begin{Proposition}\label{prop:Delta}
For $\psi \in \operatorname{End}(V)$ the following $(i)$--$(iii)$
are equivalent:
\begin{enumerate}\itemsep=0pt
\item[$(i)$] $\Delta \psi = \psi \Delta$ and
\begin{gather}\label{eq:psiTau}
\psi \tau_i \in \mathbb F \tau_{i-1}, \qquad 0 \leq i \leq N;
\end{gather}
\item[$(ii)$] $\Delta \psi = \psi \Delta$ and
\begin{gather}\label{eq:psiEta}
\psi \eta_i \in \mathbb F \eta_{i-1}, \qquad 0 \leq i \leq N;
\end{gather}
\item[$(iii)$] $\psi \in \mathcal L$.
\end{enumerate}
\end{Proposition}
\begin{proof}
(i) $\Rightarrow$ (ii) We show (\ref{eq:psiEta}). Using~(\ref{def:Delta}),
\begin{gather*}
\psi \eta_i = \psi \Delta \tau_i = \Delta \psi \tau_i \in \mathbb F \Delta \tau_{i-1} = \mathbb F \eta_{i-1}.
\end{gather*}
(ii) $\Rightarrow$ (i) Similar to the proof of (i)~$\Rightarrow$~(ii).
(i), (ii) $\Rightarrow$ (iii) By Definitions~\ref{def:dlowering},~\ref{def:dl}.
(iii) $\Rightarrow$ (i) We have~(\ref{eq:psiTau}) by Definition~\ref{def:dl}. We show $\Delta \psi = \psi \Delta$.
By Lemma~\ref{lem:normBasis} we may assume that $\psi$ is normalized.
The vector space $V$ has a basis $\lbrace \tau_i \rbrace_{i=0}^N$.
By Lemma~\ref{lem:charNorm}
we have
\begin{gather*}
\psi \tau_i = \vartheta_i \tau_{i-1}, \qquad
\psi \eta_i = \vartheta_i \eta_{i-1}, \qquad 0 \leq i \leq N.
\end{gather*}
So for $0 \leq i \leq N$,
\begin{gather*}
\Delta \psi \tau_i = \vartheta_i \Delta \tau_{i-1} =
\vartheta_i \eta_{i-1} = \psi \eta_i = \psi \Delta \tau_i.
\end{gather*}
By these comments $\Delta \psi = \psi \Delta$.
\end{proof}

We introduce some notation. For $0 \leq i\leq j\leq N$
define
\begin{gather}
\left[\begin{matrix} j \\
i \end{matrix} \right]_\vartheta =
\frac{
\vartheta_j \vartheta_{j-1} \cdots \vartheta_{j-i+1}}
{\vartheta_1 \vartheta_2 \cdots \vartheta_i}.
\label{eq:brackij}
\end{gather}
Note that
\begin{gather*}
\left[\begin{matrix} j \\
i \end{matrix} \right]_\vartheta =
\left[\begin{matrix} j \\
j-i \end{matrix} \right]_\vartheta,
 \qquad 0 \leq i \leq j \leq N.
\end{gather*}

\begin{Proposition}\label{prop:tfae4}
The following $(i)$--$(iii)$ are equivalent:
\begin{enumerate}\itemsep=0pt
\item[$(i)$] $\mathcal L \not=0$;
\item[$(ii)$] for $0 \leq j\leq N$,
\begin{gather}
\eta_j = \sum_{i=0}^j \eta_{j-i} (a_0)
\left[\begin{matrix} j \\
i \end{matrix} \right]_\vartheta \tau_i;
\label{eq:etaSum}
\end{gather}
\item[$(iii)$] for $0 \leq j\leq N$,
\begin{gather*}
\tau_j = \sum_{i=0}^j \tau_{j-i} (b_0)
\left[\begin{matrix} j \\
i \end{matrix} \right]_\vartheta \eta_i.
\end{gather*}
\end{enumerate}
Assume that $(i)$--$(iii)$ hold, and let
$\psi \in \mathcal L$ be normalized. Then
\begin{gather}
\label{eq:DSum}
\Delta = \sum_{i=0}^N \frac{\eta_i(a_0)}{\vartheta_1 \vartheta_2 \cdots\vartheta_i} \psi^i,\\
\label{eq:DiSum}
\Delta^{-1}= \sum_{i=0}^N \frac{\tau_i(b_0)}{\vartheta_1 \vartheta_2 \cdots\vartheta_i} \psi^i.
\end{gather}
\end{Proposition}
\begin{proof}
(i) $\Rightarrow$ (ii) Let $j$ be given. There exist scalars $\lbrace \alpha_i \rbrace_{i=0}^j$
in $\mathbb F$ such that
\begin{gather}
\eta_j = \sum_{i=0}^j \alpha_i \tau_i.
\label{eq:etaTau}
\end{gather}
For $0 \leq i \leq j$ we show
\begin{gather}
\alpha_i = \eta_{j-i}(a_0)
\left[\begin{matrix} j \\
i \end{matrix} \right]_\vartheta. \label{eq:goalalpha}
\end{gather}
Let $\psi \in \mathcal L$ be normalized. In
(\ref{eq:etaTau}) we apply $\psi^i$ to each side,
and evaluate the
result using
Lemma~\ref{lem:charNorm}.
We then set $x=a_0$, and use the fact that
$\tau_0=1$ and $\tau_k (a_0)=0$ for $1 \leq k \leq N$.
By these comments
\begin{gather*}
\eta_{j-i} (a_0)
\vartheta_j \vartheta_{j-1} \cdots \vartheta_{j-i+1}
= \alpha_i \vartheta_1 \vartheta_2 \cdots \vartheta_i.
\end{gather*}
From this equation we obtain
(\ref{eq:goalalpha}).
By (\ref{eq:etaTau}),
(\ref{eq:goalalpha})
we obtain (\ref{eq:etaSum}), so (ii) holds.

(ii) $\Rightarrow$ (i) Define
 $\psi \in \operatorname{End}(V)$ such that
 $\psi \tau_i = \vartheta_i \tau_{i-1}$ for $0 \leq i \leq N$.
 We have $\psi \not=0$ since $N\geq 1$ and $\vartheta_1=1$,
 $\tau_0=1$. We show $\psi \in \mathcal L$.
To do this, it is convenient to first show that $\psi$ satisfies~(\ref{eq:DSum}).
To obtain~(\ref{eq:DSum}),
for $0 \leq j \leq N$ we apply
each side of~(\ref{eq:DSum}) to $\tau_j$.
Concerning the left-hand side of~(\ref{eq:DSum}),
we have $\Delta \tau_j = \eta_j$
by~(\ref{def:Delta}).
Concerning the right-hand side of~(\ref{eq:DSum}),
\begin{align*}
\sum_{i=0}^j \frac{\eta_i(a_0)}{\vartheta_1 \vartheta_2\cdots \vartheta_i}
\psi^i \tau_j &=
\sum_{i=0}^j \frac{\eta_i(a_0)
\vartheta_j \vartheta_{j-1} \cdots \vartheta_{j-i+1}
}{\vartheta_1 \vartheta_2\cdots \vartheta_i}
 \tau_{j-i}
=
\sum_{i=0}^j \eta_i(a_0)
\left[\begin{matrix} j \\
i \end{matrix} \right]_\vartheta
 \tau_{j-i}
\\
 &=
\sum_{i=0}^j \eta_{j-i}(a_0)
\left[\begin{matrix} j \\
i \end{matrix} \right]_\vartheta
 \tau_{i}
=\eta_j.
\end{align*}
We have shown that each side of
(\ref{eq:DSum}) sends $\tau_j \mapsto \eta_j$ for $0 \leq j \leq N$.
Therefore (\ref{eq:DSum}) holds. By~(\ref{eq:DSum}), $\Delta$ is a polynomial in $\psi$.
Consequently $\psi \Delta = \Delta \psi$.
The map $\psi$ satisfies
Proposition~\ref{prop:Delta}(i),
so $\psi \in \mathcal L$ by
Proposition~\ref{prop:Delta}(i),~(iii).
Therefore $\mathcal L \not=0$.

(i) $\Leftrightarrow$ (iii)
Interchange the roles of
$\lbrace a_i\rbrace_{i=0}^{N-1}$,
$\lbrace b_i\rbrace_{i=0}^{N-1}$ in the proof of
(i) $\Leftrightarrow$ (ii).

Now assume that (i)--(iii) hold. We saw in
the proof of
(ii) $\Rightarrow$ (i) that~(\ref{eq:DSum}) holds.
Interchanging the roles of
$\lbrace a_i\rbrace_{i=0}^{N-1}$,
$\lbrace b_i\rbrace_{i=0}^{N-1}$ in that proof, we see that~(\ref{eq:DiSum}) holds.
\end{proof}

Later in the paper, we will obtain necessary and sufficient
conditions for the data~(\ref{eq:data}) to satisfy conditions
(i)--(iii)
in Proposition~\ref{prop:tfae4}; our result is Theorem~\ref{thm:hr}.
In order to motivate this result,
we look at some examples of double lowering data.
This will be done in the next section.

\section{First examples of double lowering data}\label{section6}

We continue to discuss the double lowering
space $\mathcal L$ for
the data (\ref{eq:data}). In this section we give
three assumptions under which this data is double lowering.
Under each assumption we describe the polynomials
$\lbrace \tau_i\rbrace_{i=0}^{N-1}$,
$\lbrace \eta_i\rbrace_{i=0}^{N-1}$ from
(\ref{eq:tau}),
(\ref{eq:eta}),
the parameters $\lbrace \vartheta_i \rbrace_{i=0}^{N}$
from Definition
\ref{def:vartheta},
and the
map $\Delta$ from
(\ref{def:Delta}).

As a warmup, we examine the condition (\ref{eq:etaSum}) for some small values of $j$.

\begin{Lemma}\label{lem:Prop}
The following $(i)$--$(iv)$ hold.
\begin{enumerate}\itemsep=0pt
\item[$(i)$] $\eta_0 = \tau_0$.
\item[$(ii)$]
 $\eta_1 = \eta_1(a_0)\tau_0 + \tau_1$.
\item[$(iii)$] For $N\geq 2$,
\begin{gather*}
\eta_2 = \eta_2(a_0)\tau_0 + \vartheta_2 \eta_1(a_0) \tau_1 + \tau_2.
\end{gather*}
\item[$(iv)$] For $N\geq 3$,
\begin{gather*}
\eta_3 = \eta_3(a_0)\tau_0 +
\vartheta_3 \eta_2(a_0) \tau_1+
\vartheta_3 \eta_1(a_0) \tau_2+
\tau_3 +(x - a_0) \varepsilon,
\end{gather*}
where
\begin{gather*}
\varepsilon =
(b_0 - a_1)(b_2-a_1)-
(a_0 - b_1)(a_2-b_1).
\end{gather*}
\end{enumerate}
\end{Lemma}
\begin{proof} To verify these equations, evaluate the terms using~(\ref{eq:tau}), (\ref{eq:eta}) and Definition~\ref{def:vartheta}.
\end{proof}

\begin{Lemma}\label{lem:n2}
Assume that $N\leq 2$. Then $\mathcal L\not=0$.
\end{Lemma}
\begin{proof} By Proposition \ref{prop:tfae4}(i),~(ii) and Lemma~\ref{lem:Prop}.
\end{proof}

\begin{Lemma}\label{lem:N3}
Assume that $N=3$. Then $\mathcal L\not=0$ if and only if
\begin{gather*}
(a_0 - b_1)(a_2-b_1) = (b_0 - a_1)(b_2-a_1).
\end{gather*}
\end{Lemma}
\begin{proof} By Proposition \ref{prop:tfae4}(i),(ii) and Lemma~\ref{lem:Prop}.
\end{proof}

\begin{Lemma}\label{lem:five}
Assume that $a_{i-1} = b_i$ for $1 \leq i \leq N-1$. Then the following {\rm (i)--(v)} hold:
\begin{enumerate}\itemsep=0pt
\item[$(i)$] $\mathcal L\not=0$;
\item[$(ii)$] $\eta_i = (x-b_0) \tau_{i-1}$ for $1 \leq i \leq N$;
\item[$(iii)$] $\eta_i(a_0)=0$ for $2 \leq i \leq N$;
\item[$(iv)$] $\vartheta_i = \frac{a_{i-1}-b_0}{a_0-b_0}$ for $
1 \leq i \leq N$;
\item[$(v)$] $\Delta = I + (a_0-b_0)\psi$, where $\psi \in \mathcal L$ is normalized.
\end{enumerate}
\end{Lemma}
\begin{proof} (ii)~By (\ref{eq:tau}), (\ref{eq:eta}).
(iii)~By (ii) and since $\tau_j(a_0)=0 $ for $1 \leq j \leq N$.
(iv)~Use Definition~\ref{def:vartheta}.
(i)~Apply Proposition~\ref{prop:tfae4}(i),~(ii).
(v)~By~(\ref{eq:DSum}) and~(iii) above.
\end{proof}

\begin{Lemma}\label{lem:fivedual}
Assume that $a_{i} = b_{i-1}$ for $1 \leq i \leq N-1$.
Then the following $(i)$--$(v)$ hold:
\begin{enumerate}\itemsep=0pt
\item[$(i)$] $\mathcal L\not=0$;
\item[$(ii)$] $\tau_i = (x-a_0) \eta_{i-1}$ for $1 \leq i \leq N$;
\item[$(iii)$] $\tau_i(b_0)=0$ for $2 \leq i \leq N$;
\item[$(iv)$] $\displaystyle{
\vartheta_i = \frac{a_{0}-b_{i-1}}{a_0-b_0}}$ for $ 1 \leq i \leq N$;
\item[$(v)$] $\Delta^{-1} = I + (b_0-a_0)\psi$, where $\psi \in \mathcal L$
is normalized.
\end{enumerate}
\end{Lemma}
\begin{proof} Interchange the roles of
$\lbrace a_i\rbrace_{i=0}^{N-1}$,
$\lbrace b_i\rbrace_{i=0}^{N-1}$ in Lemma~\ref{lem:five}.
\end{proof}

For the rest of this section,
assume that $N\geq 2$. Also for the rest of this section,
fix $\theta \in \mathbb F$ and assume
\begin{gather}
a_0 \not=\theta, \qquad b_0 \not=\theta;\label{eq:c1}\\
a_i = \theta, \qquad b_i = \theta, \qquad 1 \leq i \leq N-2;\label{eq:c2}\\
\frac{\theta- a_{N-1}}{\theta-b_0} =
\frac{\theta- b_{N-1}}{\theta-a_0} \qquad \text{if} \ N\not=\infty.\label{eq:c3}
\end{gather}
Using Definition~\ref{def:vartheta},
\begin{gather*}
\vartheta_i = 1, \qquad 1 \leq i \leq N-1
\end{gather*}
and for $N\not=\infty$,
\begin{gather*}
a_{N-1} = b_0 +\vartheta_N (\theta-b_0), \qquad
b_{N-1} = a_0 +\vartheta_N (\theta-a_0).
\end{gather*}
Using~(\ref{eq:brackij}),
\begin{gather*}
\left[\begin{matrix} j \\
i \end{matrix} \right]_\vartheta = 1,
 \qquad 0 \leq i\leq j\leq N-1
\end{gather*}
and for $N\not=\infty$,
\begin{gather*}
\left[\begin{matrix} N \\
i \end{matrix} \right]_\vartheta = \vartheta_N,
 \qquad 1 \leq i \leq N-1.
\end{gather*}
For $0 \leq i \leq N$ the polynomials $\tau_i$, $\eta_i$
are described in the table below:
\begin{center}
\begin{tabular}[t]{c|cc}
 $i$ & $\tau_i$ & $\eta_i$
\\
\hline
$0$ &
$1$ & $1$
\\
$1 \leq i \leq N-1$ &
$(x-a_0)(x-\theta)^{i-1}$ &
$(x-b_0)(x-\theta)^{i-1}$
\\
$N$ &
$(x-a_0)(x-\theta)^{N-2}(x-a_{N-1})$ &
$(x-b_0)(x-\theta)^{N-2}(x-b_{N-1})$
\end{tabular}
\end{center}

For $1 \leq i \leq N$ the values of $\eta_i - \tau_i$
and $\eta_i(a_0)$ are described in the table below:
\begin{center}
\begin{tabular}[t]{c|cc}
 $i$ & $\eta_i - \tau_i$ & $\eta_i(a_0)$
\\
\hline
$1 \leq i \leq N-1$ &
$(a_0-b_0)(x-\theta)^{i-1}$ &
$(a_0-b_0)(a_0-\theta)^{i-1}$
\\
$N$ &
$\vartheta_N (a_0-b_0)(x-\theta)^{N-1}$ &
$\vartheta_N (a_0-b_0)(a_0-\theta)^{N-1}$
\end{tabular}
\end{center}

\begin{Lemma}\label{lem:deg} Under assumptions
\eqref{eq:c1}--\eqref{eq:c3} the following $(i)$--$(iii)$ hold:
\begin{enumerate}\itemsep=0pt
\item[$(i)$] $\mathcal L\not=0$;
\item[$(ii)$] $\Delta =\frac{I + (\theta- b_0)\psi}{I + (\theta-a_0)\psi}$;
\item[$(iii)$] $\Delta^{-1} =\frac{I + (\theta- a_0)\psi}{I + (\theta-b_0)\psi}$.
\end{enumerate}
In the above lines $\psi \in \mathcal L$ is normalized.
\end{Lemma}
\begin{proof} (i)
We invoke Proposition~\ref{prop:tfae4}(i),~(ii).
For $0 \leq j \leq N$ we verify~(\ref{eq:etaSum}).
We may assume that $2 \leq j \leq N$; otherwise we are done
by Lemma~\ref{lem:Prop}. For $N \not=\infty$ we separate
the cases $2\leq j\leq N-1$ and $j=N$.
It suffices to show that
\begin{gather}
\eta_j = \tau_j + \eta_j(a_0)+
\sum_{i=1}^{j-1} \eta_{j-i}(a_0)\tau_i,
\qquad 2 \leq j \leq N-1,\label{eq:jnotN}
\\
\eta_N = \tau_N + \eta_N(a_0)+ \vartheta_N \sum_{i=1}^{N-1}
\eta_{N-i}(a_0)\tau_i, \qquad
\text{if} \ N\not=\infty.
\label{eq:jisN}
\end{gather}
For $2 \leq j \leq N$ the values of $\eta_j-\tau_j$ and
$\eta_j(a_0)$ are given in the table above the lemma statement.
Also for $2 \leq j \leq N$,{\samepage
\begin{align*}
\sum_{i=1}^{j-1} \eta_{j-i}(a_0)\tau_i
&=(a_0-b_0)(x-a_0) \sum_{i=1}^{j-1} (a_0-\theta)^{j-i-1}(x-\theta)^{i-1}\\
&=(a_0-b_0)(x-a_0)(a_0-\theta)^{j-2} \sum_{k=0}^{j-2}\left(
\frac{x-\theta}{a_0-\theta}\right)^k\\
&= (a_0-b_0)(x-\theta)^{j-1} - (a_0-b_0)(a_0-\theta)^{j-1}.
\end{align*}
Using the above comments we routinely verify~(\ref{eq:jnotN}),~(\ref{eq:jisN}).}

(ii) We will verify the equation by showing
that the two sides agree on $V_j$ for $2 \leq j \leq N$.
Using~(\ref{eq:DSum}) and $\psi^{j+1}V_j=0$ we see that on $V_j$,
\begin{align*}
\Delta - I &= \sum_{i=1}^j \frac{\eta_i(a_0)}{\vartheta_1 \vartheta_2\cdots
\vartheta_i}\psi^i
= \frac{\eta_j(a_0)}{\vartheta_j} \psi^j + \sum_{i=1}^{j-1}
\eta_i(a_0) \psi^i
= (a_0-b_0) \psi \sum_{k=0}^{j-1} (a_0-\theta)^k \psi^k
\\
&= (a_0-b_0)\psi \frac{ I - (a_0-\theta)^j \psi^j}{I - (a_0-\theta) \psi}
= \frac{(a_0-b_0)\psi}{I - (a_0-\theta) \psi}.
\end{align*}
The result follows.
(iii)~By (ii) above.
\end{proof}

We just gave some examples of double lowering data. There is another example that is somewhat more involved; it will be described later in the paper.

\section{Extending the data}\label{section7}

Throughout this section, we assume that
$N$ is an integer at least~2.
Recall the data $\lbrace a_i \rbrace_{i=0}^{N-1}$,
$\lbrace b_i \rbrace_{i=0}^{N-1}$ from~(\ref{eq:data}), and assume that this data is double lowering.
Let $a_N, b_N \in \mathbb F$ satisfy
\begin{gather*}
a_0 + a_1 + \cdots + a_N \not=
b_0 + b_1 + \cdots + b_N,
\end{gather*}
giving data
\begin{gather}
\lbrace a_i \rbrace_{i=0}^{N}, \qquad
\lbrace b_i \rbrace_{i=0}^{N}.
\label{eq:dataExt}
\end{gather}
In this section we obtain necessary
and sufficient conditions on $a_N$, $b_N$
for the data (\ref{eq:dataExt}) to be double lowering.
By~(\ref{eq:xrec}),
\begin{gather*}
\tau_{N+1} = (x-a_N)\tau_N, \qquad
\eta_{N+1} = (x-b_N)\eta_N.
\end{gather*}
\begin{Lemma}\label{lem:feas3}
The following $(i)$--$(iii)$ are equivalent:
\begin{enumerate}\itemsep=0pt
\item[$(i)$]
the data \eqref{eq:dataExt}
is double lowering;
\item[$(ii)$]
we have
\begin{gather}
\eta_{N+1} = \sum_{i=0}^{N+1} \eta_{N-i+1} (a_0)
\left[\begin{matrix} N+1 \\
i \end{matrix} \right]_\vartheta \tau_i;\label{eq:etaSumExt}
\end{gather}
\item[$(iii)$] we have
\begin{gather*}
\tau_{N+1} = \sum_{i=0}^{N+1} \tau_{N-i+1} (b_0)
\left[\begin{matrix} N+1 \\
i \end{matrix} \right]_\vartheta \eta_i.
\end{gather*}
\end{enumerate}
\end{Lemma}
\begin{proof} By Proposition~\ref{prop:tfae4}.
\end{proof}

\begin{Lemma}\label{lem:etalt}
We have
\begin{gather}
\eta_{N+1}= \sum_{i=0}^N \eta_{N-i}(a_0)
\left[\begin{matrix} N \\
i \end{matrix} \right]_\vartheta
(a_i-b_N)\tau_i
 + \sum_{i=1}^{N+1} \eta_{N-i+1}(a_0)
\left[\begin{matrix} N \\
i-1 \end{matrix} \right]_\vartheta \tau_i,
\label{eq:etaalt}
\\
\tau_{N+1}= \sum_{i=0}^N \tau_{N-i}(b_0)
\left[\begin{matrix} N \\
i \end{matrix} \right]_\vartheta
(b_i-a_N)\eta_i
 + \sum_{i=1}^{N+1} \tau_{N-i+1}(b_0)
\left[\begin{matrix} N \\
i-1 \end{matrix} \right]_\vartheta \eta_i.\label{eq:taualt}
\end{gather}
\end{Lemma}
\begin{proof} We show
(\ref{eq:etaalt}). Using Proposition
\ref{prop:tfae4},
\begin{align*}
\eta_{N+1} & = (x-b_N)\eta_N
= (x-b_N)\sum_{i=0}^N \eta_{N-i}(a_0)
\left[\begin{matrix} N \\
i \end{matrix} \right]_\vartheta \tau_i
\\
&= \sum_{i=0}^N \eta_{N-i}(a_0)
\left[\begin{matrix} N \\
i \end{matrix} \right]_\vartheta
\bigl((a_i-b_N)\tau_i+\tau_{i+1}\bigr)
\\
&= \sum_{i=0}^N \eta_{N-i}(a_0)
\left[\begin{matrix} N \\
i \end{matrix} \right]_\vartheta
(a_i-b_N)\tau_i
 + \sum_{i=1}^{N+1} \eta_{N-i+1}(a_0)
\left[\begin{matrix} N \\
i-1 \end{matrix} \right]_\vartheta
\tau_i.
\end{align*}
Line~(\ref{eq:taualt}) is similarly obtained.
\end{proof}

\begin{Proposition}\label{prop:feas3}
The following
$(i)$--$(iii)$ are equivalent:
\begin{enumerate}\itemsep=0pt
\item[$(i)$] the data~\eqref{eq:dataExt}
is double lowering;
\item[$(ii)$] for $0 \leq i \leq N-1$ such that $\eta_i(a_0)\not=0$,
\begin{gather*}
(a_0+ \cdots + a_i -b_0-\cdots - b_i)(a_{N-i}-b_N)\\
\qquad{} = (a_0-b_i)(a_{N-i}+ \cdots + a_N-b_{N-i}- \cdots -b_N);
\end{gather*}
\item[$(iii)$] for $0 \leq i \leq N-1$ such that $\tau_i(b_0)\not=0$,
\begin{gather*}
(b_0+ \cdots + b_i -a_0-\cdots - a_i)(b_{N-i}-a_N)\\
\qquad{} = (b_0-a_i)(b_{N-i}+ \cdots + b_N-a_{N-i}- \cdots -a_N).
\end{gather*}
\end{enumerate}
\end{Proposition}
\begin{proof}
(i) $\Leftrightarrow$ $(ii)$
We invoke Lemma~\ref{lem:feas3}(i),~(ii).
Subtract~(\ref{eq:etaSumExt})
from~(\ref{eq:etaalt}) to obtain an equation
$0 = \sum_{i=1}^N d_i \tau_i$ where
\begin{gather}\label{eq:di}
d_i=
 \eta_{N-i}(a_0)
\left[\begin{matrix} N \\
i \end{matrix} \right]_\vartheta
(a_i-b_N)
+
 \eta_{N-i+1}(a_0)
\left[\begin{matrix} N \\
i-1 \end{matrix} \right]_\vartheta
-
\eta_{N-i+1}(a_0)
\left[\begin{matrix} N+1 \\
i \end{matrix} \right]_\vartheta
\end{gather}
for $1 \leq i \leq N$.
Note that~(\ref{eq:etaSumExt}) holds iff
$0 = \sum_{i=1}^N d_i \tau_i$ iff
$d_i=0$ for $1\leq i \leq N$.
For $1 \leq i \leq N$
we simplify~(\ref{eq:di}) using
\begin{gather*}
\eta_{N-i+1}(a_0)= \eta_{N-i}(a_0)(a_0-b_{N-i})
\end{gather*}
and
\begin{gather*}
\left[\begin{matrix} N \\
i-1 \end{matrix} \right]_\vartheta
=
\left[\begin{matrix} N \\
i \end{matrix} \right]_\vartheta \frac{\vartheta_i}{\vartheta_{N-i+1}},
 \qquad
\left[\begin{matrix} N+1 \\
i \end{matrix} \right]_\vartheta
=
\left[\begin{matrix} N \\
i \end{matrix} \right]_\vartheta \frac{\vartheta_{N+1}}{\vartheta_{N-i+1}}.
\end{gather*}
We find that $d_i$ is equal to
\begin{gather*}
\frac{\eta_{N-i}(a_0)}{\vartheta_{N-i+1}}
\left[\begin{matrix} N \\
i \end{matrix} \right]_\vartheta
\end{gather*}
times
\begin{gather*}
(a_{i}-b_N)\vartheta_{N-i+1}
+(a_0-b_{N-i})(\vartheta_{i}-\vartheta_{N+1})
\end{gather*}
for $1 \leq i \leq N$.
Therefore,
(\ref{eq:etaSumExt}) holds if and only if
\begin{gather*}
\eta_{N-i}(a_0)=0
\qquad \text{or}\qquad
(a_{i}-b_N)\vartheta_{N-i+1} = (a_0-b_{N-i})(\vartheta_{N+1}-\vartheta_{i})
\end{gather*}
for $1 \leq i \leq N$.
Replacing $i$ by $N-i$,
we see that~(\ref{eq:etaSumExt}) holds if and only if
\begin{gather*}
\eta_{i}(a_0)=0
\qquad \text{or}\qquad
(a_{N-i}-b_N)\vartheta_{i+1} = (a_0-b_{i})(\vartheta_{N+1}-\vartheta_{N-i})
\end{gather*}
for $0 \leq i \leq N-1$.
The result follows in view of
Definition~\ref{def:vartheta}.

(i) $\Leftrightarrow$ (iii) Similar to the proof of~(i)~$\Leftrightarrow$~(ii).
\end{proof}

Our next general goal is to solve the equations
in Proposition~\ref{prop:feas3}(ii), (iii). The main solution will involve a type
of sequence, said to be recurrent.

\section{Recurrent sequences}\label{section8}

 Throughout this section
let $n$ denote an integer at least 2, or $\infty$.
let $\lbrace a_i\rbrace_{i=0}^n$ denote
 scalars in $\mathbb F$.

\begin{Definition}[{see \cite[Definition~8.2]{2lintrans}}]\label{lem:beginthreetermS99}
Let $\beta$, $\gamma$, $\varrho$ denote scalars
in $\mathbb F$.
\begin{enumerate}\itemsep=0pt
\item[(i)] The sequence
 $\lbrace a_i\rbrace_{i=0}^n$
is said to be {\it $(\beta,\gamma,\varrho)$-recurrent} whenever
\begin{equation}
a^2_{i-1}-\beta a_{i-1}a_i+a^2_i
-\gamma (a_{i-1} +a_i)=\varrho
\label{eq:varrhothreetermS99}
\end{equation}
 for
$1 \leq i \leq n$.
\item[(ii)] The sequence
 $\lbrace a_i\rbrace_{i=0}^n$
is said to be {\it $(\beta,\gamma)$-recurrent} whenever
\begin{equation}
a_{i-1}-\beta a_i+a_{i+1}=\gamma
\label{eq:gammathreetermS99}
\end{equation}
 for
$1 \leq i \leq n-1$.
\item[(iii)] The sequence
 $\lbrace a_i\rbrace_{i=0}^n$
is said to be {\it $\beta$-recurrent} whenever
\begin{equation}
a_{i-2}-(\beta+1)a_{i-1}+(\beta +1)a_i -a_{i+1}
\label{eq:betarecS99}
\end{equation}
is zero for
$2 \leq i \leq n-1$.
\item[(iv)] The sequence
 $\lbrace a_i\rbrace_{i=0}^n$
is said to be {\it recurrent} whenever there exists
$\beta \in \mathbb F$ such that
 $\lbrace a_i\rbrace_{i=0}^n$ is $\beta$-recurrent.
\end{enumerate}
\end{Definition}

\begin{Lemma}\label{lem:recvsbrecS99}
The following are equivalent:
\begin{enumerate}\itemsep=0pt
\item[$(i)$] the sequence
 $\lbrace a_i\rbrace_{i=0}^n$
is recurrent;
\item[$(ii)$]
there exists $\beta \in \mathbb F$ such that
 $\lbrace a_i\rbrace_{i=0}^n$
 is $\beta$-recurrent.
\end{enumerate}
\end{Lemma}
\begin{proof} By Definition~\ref{lem:beginthreetermS99}.
\end{proof}

\begin{Lemma}\label{lem:brecvsbgrecS99}
For $\beta \in \mathbb F$ the following are equivalent:
\begin{enumerate}\itemsep=0pt
\item[$(i)$] the sequence
 $\lbrace a_i\rbrace_{i=0}^n$
is $\beta$-recurrent;
\item[$(ii)$] there exists $\gamma \in \mathbb F$ such that
 $\lbrace a_i\rbrace_{i=0}^n$ is $(\beta,\gamma)$-recurrent.
\end{enumerate}
\end{Lemma}
\begin{proof}
(i) $\Rightarrow$ (ii) For $2\leq i \leq n-1$, the expression
(\ref{eq:betarecS99}) is zero by assumption,
so
\begin{gather*}
a_{i-2}-\beta a_{i-1}+a_i =
a_{i-1}-\beta a_i+a_{i+1}.
\end{gather*}
The left-hand side of~(\ref{eq:gammathreetermS99}) is independent of $i$, and
the result follows.

(ii) $\Rightarrow$ (i) For $2\leq i \leq n-1$,
subtract the equation~(\ref{eq:gammathreetermS99}) at $i$ from the corresponding equation
obtained by replacing $i$ by $i-1$, to find~(\ref{eq:betarecS99}) is zero.
\end{proof}

\begin{Lemma}\label{lem:bgrecvsbgdrecS99}
The following $(i)$, $(ii)$ hold
for all $\beta, \gamma \in \mathbb F$.
\begin{enumerate}\itemsep=0pt
\item[$(i)$] Suppose
 $\lbrace a_i\rbrace_{i=0}^n$
is $(\beta,\gamma)$-recurrent. Then
there exists $\varrho \in \mathbb F$ such that
 $\lbrace a_i\rbrace_{i=0}^n$
 is $(\beta,\gamma,\varrho)$-recurrent.
\item[$(ii)$] Suppose
 $\lbrace a_i\rbrace_{i=0}^n$
is $(\beta,\gamma,\varrho)$-recurrent, and that $a_{i-1}\not=a_{i+1}$
for $1 \leq i\leq n-1$. Then
 $\lbrace a_i\rbrace_{i=0}^n$
 is $(\beta,\gamma)$-recurrent.
\end{enumerate}
\end{Lemma}

\begin{proof}Let $p_i$ denote the expression on the left in~(\ref{eq:varrhothreetermS99}),
and observe
 \begin{gather*}
p_i-p_{i+1} = (a_{i-1}-a_{i+1})(a_{i-1}-\beta a_i +a_{i+1} - \gamma)
\end{gather*}
for $1 \leq i \leq n-1$. Assertions (i), (ii) are both routine consequences of this.
\end{proof}

\begin{Definition}\label{def:pt}
Assume that $\lbrace a_i \rbrace_{i=0}^n$ is recurrent.
By a {\it parameter triple} for
$\lbrace a_i \rbrace_{i=0}^n$ we mean a 3-tuple
 $\beta$, $\gamma$, $\varrho$ of scalars in $\mathbb F$
such that $\lbrace a_i \rbrace_{i=0}^n$ is
$(\beta,\gamma)$-recurrent and $(\beta, \gamma, \varrho)$-recurrent.
\end{Definition}

Note that a recurrent sequence has at least one parameter triple.

\section{Recurrent sequences in closed form }\label{section9}

In this section, we describe the recurrent sequences in closed form.
Let $n$ denote an integer at least~2, or~$\infty$.

\begin{Lemma}[{see \cite[Lemma~9.2]{2lintrans}}]\label{lem:closedformthreetermS99}
 The recurrent sequences
 $\lbrace a_i\rbrace_{i=0}^n$
are described in the table below:
\begin{center}
\begin{tabular}[t]{c|c|c}
 {\rm case} & $a_i$ & {\rm comments}
\\
\hline
{\rm I} & $ \alpha_1 + \alpha_2 q^i + \alpha_3 q^{-i} $
&
$q\not\in \lbrace 0,1,-1\rbrace$
 \\
{\rm II} &
$ \alpha_1 + \alpha_2 i + \alpha_3 \binom{i}{2} $ &
\\
{\rm III} &
$ \alpha_1 + \alpha_2 (-1)^i + \alpha_3 i(-1)^i$ &
${\rm char}(\mathbb F)\not=2$
\\
\end{tabular}
\end{center}

In the above table $q$, $\alpha_1$, $\alpha_2$, $\alpha_3$ are scalars in~$\mathbb F$.
\end{Lemma}

\begin{Lemma}\label{lem:bgv}
The following scalars $\beta$, $\gamma$, $\varrho$ give a parameter triple
for the recurrent sequence $\lbrace a_i \rbrace_{i=0}^n$ in
Lemma~{\rm \ref{lem:closedformthreetermS99}}.

Case I:
\begin{gather*}
\beta = q + q^{-1},
\qquad
\gamma = - \alpha_1 (q-1)^2 q^{-1},
\qquad
\varrho = \alpha_1^2(q-1)^2 q^{-1} - \alpha_2 \alpha_3 \big(q-q^{-1}\big)^2.
\end{gather*}

Case II:
\begin{gather*}
\beta = 2, \qquad \gamma = \alpha_3, \qquad
\varrho = \alpha_2^2 - \alpha_2 \alpha_3 - 2 \alpha_1 \alpha_3.
\end{gather*}

Case III:
\begin{gather*}
\beta = -2,
\qquad
\gamma = 4 \alpha_1,
\qquad
\varrho = \alpha_3^2 - 4 \alpha_1^2.
\end{gather*}
\end{Lemma}
\begin{proof} This is routinely checked using Definition~\ref{def:pt}.
\end{proof}

\begin{Lemma}\label{lem:asum}
Referring to Lemma~{\rm \ref{lem:closedformthreetermS99}}, for $0 \leq i \leq n+1$
the sum $a_0+a_1+ \cdots + a_{i-1}$ is given in the table below:
\begin{center}
\begin{tabular}[t]{c|c}
case & $a_0+a_1+ \cdots + a_{i-1}$
\\
\hline
I & $\alpha_1 i + \alpha_2 \frac{1-q^i}{1-q} + \alpha_3
\frac{1-q^{-i}}{1-q^{-1}} $
 \\
II &
$\alpha_1 i + \alpha_2 \binom{i}{2} + \alpha_3 \binom{i}{3} $
\\
III, $i$ even &
 $\alpha_1 i - \alpha_3 i/2$
\\
III, $i$ odd &
$\alpha_1 i + \alpha_2 + \alpha_3(i-1)/2$
\end{tabular}
\end{center}
\end{Lemma}
\begin{proof} Use induction on $i$.
\end{proof}

\begin{Note}Referring to Case~III of the above table,
the subcases $i$ even and $i$ odd can be handled
in the following uniform way. For $0 \leq i \leq n+1$,
\begin{gather*}
a_0 + a_1+\cdots + a_{i-1} = \frac{2\alpha_2-\alpha_3 +4\alpha_1 i
+ (\alpha_3-2\alpha_2)(-1)^i -2\alpha_3 i (-1)^i}{4}.
\end{gather*}
\end{Note}

\section{Twin recurrent sequences}\label{section10}

Let $n$ denote an integer at least 2, or $\infty$.
Let $\lbrace a_i\rbrace_{i=0}^n$, $\lbrace b_i\rbrace_{i=0}^n$ denote scalars in~$\mathbb F$.

\begin{Definition}\label{def:twin}
Assume that the sequences
$\lbrace a_i\rbrace_{i=0}^n$,
$\lbrace b_i\rbrace_{i=0}^n$ are recurrent.
These sequences are called {\it twins}
whenever they have a parameter triple in common.
\end{Definition}

\begin{Lemma}\label{lem:twin}
 Assume that the sequences $\lbrace a_i\rbrace_{i=0}^n$,
$\lbrace b_i\rbrace_{i=0}^n$ are recurrent.
These sequences are twins if and only if
they are related in the following way:

Case I:
\begin{gather*}
 a_i= \alpha_1 + \alpha_2 q^i + \alpha_3 q^{-i},
\qquad
 b_i= \alpha_1 + \alpha'_2 q^i + \alpha'_3 q^{-i},
\qquad
 \alpha_2' \alpha_3' = \alpha_2 \alpha_3.
\end{gather*}

Case II:
\begin{gather*}
 a_i= \alpha_1 + \alpha_2 i + \alpha_3 \binom{i}{2},
\qquad
 b_i= \alpha'_1 + \alpha'_2 i + \alpha_3 \binom{i}{2},
\\
(\alpha_2-\alpha'_2)(\alpha_2+\alpha'_2-\alpha_3)=
2(\alpha_1-\alpha'_1)\alpha_3.
\end{gather*}

Case III:
\begin{gather*}
 a_i= \alpha_1 + \alpha_2 (-1)^i + \alpha_3 i(-1)^i,\qquad
 b_i= \alpha_1 + \alpha'_2 (-1)^i + \alpha'_3 i(-1)^i,\\
\alpha'_3= \alpha_3 \qquad \text{or} \qquad \alpha'_3=-\alpha_3 .
\end{gather*}
\end{Lemma}
\begin{proof}
First assume that the sequences $\lbrace a_i \rbrace_{i=0}^n$,
 $\lbrace b_i \rbrace_{i=0}^n$ are related in the specified way.
Then these sequences share the parameter triple
$\beta$, $\gamma$, $\varrho$ from
Lemma~\ref{lem:bgv}. Therefore these sequences are twins.
Next assume that the sequences
 $\lbrace a_i \rbrace_{i=0}^n$,
 $\lbrace b_i \rbrace_{i=0}^n$ are twins.
It follows from Definition~\ref{lem:beginthreetermS99}
and Lemma~\ref{lem:closedformthreetermS99}
that they are related in the specified way.
\end{proof}

\section{A characterization of twin recurrent sequences}\label{section11}

In this section we explain what twin recurrent
sequences have to do with the equations in
Proposition~\ref{prop:feas3}.
Let $n$ denote an integer at least 2, or $\infty$.
Let $\lbrace a_i \rbrace_{i=0}^n$,
 $\lbrace b_i \rbrace_{i=0}^n$ denote scalars in~$\mathbb F$.

\begin{Definition}\label{def:eij}
For $0 \leq i\leq j\leq n$ let $E(i,j)$ denote the equation
\begin{gather*}
(a_0+ \cdots + a_i -b_0-\cdots - b_i)
(a_{j-i}-b_j)
=
(a_{0}-b_i)(a_{j-i}+ \cdots + a_j -b_{j-i}-\cdots - b_j).
\end{gather*}
\end{Definition}
\begin{Lemma}\label{lem:triv}
The equations $E(0,j)$ and $E(j,j)$ hold for $0 \leq j \leq n$.
\end{Lemma}
\begin{proof} This is routinely checked.
\end{proof}

\begin{Proposition}\label{lem:betaGam}
Assume that the sequences $\lbrace a_i \rbrace_{i=0}^{n}$, $\lbrace b_i \rbrace_{i=0}^{n}$ are recurrent and twins.
Then $E(i,j)$ holds for $0 \leq i\leq j \leq n$.
\end{Proposition}
\begin{proof}This is routinely verified for each Case~I--III in Lemma~\ref{lem:twin}. To carry out the verification, use the
formulas in Lemma~\ref{lem:asum}.
\end{proof}

In the next two lemmas, we give some additional
solutions for the equations $E(i,j)$ in Definition~\ref{def:eij}.

\begin{Lemma}Assume that
$a_i = b_{i-1}$ for $1 \leq i \leq n$. Then
$E(i,j)$ holds for $0 \leq i\leq j \leq n$.
\end{Lemma}
\begin{proof} For
$0 \leq i\leq j \leq n$,
each side of
$E(i,j)$ is equal to
$(a_0-b_i)(a_{j-i}-b_j)$.
\end{proof}
\begin{Lemma}
Pick $\theta \in \mathbb F$ and assume
\begin{gather}
a_i = \theta,\qquad b_i = \theta, \qquad 1 \leq i \leq n-1,\nonumber\\
(\theta-a_0)(\theta-a_{n})=(\theta-b_0)(\theta-b_{n})
 \qquad \text{if} \ n\not=\infty.\label{eq:cond}
\end{gather}
Then $E(i,j)$ holds for $0 \leq i\leq j \leq n$.
\end{Lemma}
\begin{proof}By Lemma \ref{lem:triv} it suffices to verify
$E(i,j)$ for $1\leq i< j\leq n$.
Let $i,j$ be given. For $j<n$, each side of $E(i,j)$ is zero.
For
$n\not=\infty $ and $j=n$, the equation $E(i,j)$ becomes
\begin{gather*}
(a_0-b_0)(\theta-b_n)= (a_0-\theta)(a_n-b_n)
\end{gather*}
which is a reformulation of~(\ref{eq:cond}).
\end{proof}

\begin{Proposition}\label{prop:hr}
Assume that $a_1 \not=b_0$ and $a_1 \not=b_1$.
Further assume that
$E(i,j)$ holds
for $1 \leq i \leq 2$ and $i+1\leq j\leq n$. Then
the sequences $\lbrace a_i \rbrace_{i=0}^n$,
$\lbrace b_i \rbrace_{i=0}^n$ are recurrent and twins.
\end{Proposition}
\begin{proof}Using $E(1,2)$,
\begin{gather}\label{eq:twin2}
(a_0-b_1)(a_2-b_1)=
(b_0-a_1)(b_2-a_1).
\end{gather}
Since $a_1 \not=b_1$, there exists a unique pair
$\beta$, $\gamma$ of scalars in $\mathbb F$
such that
\begin{gather*}
a_0 - \beta a_1 + a_2 = \gamma,
\qquad
b_0 - \beta b_1 + b_2 = \gamma.
\end{gather*}
Using these equations we eliminate $a_2$, $b_2$ in~(\ref{eq:twin2}):
\begin{gather*}
(a_0-b_1)(\gamma + \beta a_1-a_0-b_1) =
(b_0-a_1)(\gamma + \beta b_1-b_0-a_1).
\end{gather*}
In this equation we rearrange terms to get
\begin{gather*}
a^2_{0}-\beta a_{0}a_1 + a^2_1 - \gamma (a_{0}+a_1) =
b^2_{0}-\beta b_{0}b_1 + b^2_1 - \gamma (b_{0}+b_1).
\end{gather*}
Let $\varrho$ denote this common value.
We show that each of
$\lbrace a_i \rbrace_{i=0}^n$,
$\lbrace b_i \rbrace_{i=0}^n$ is
 $(\beta, \gamma)$-recurrent
and $(\beta, \gamma,\varrho)$-recurrent.
To do this, we show that for $2\leq j \leq n$, each of
$\lbrace a_i \rbrace_{i=0}^j$,
$\lbrace b_i \rbrace_{i=0}^j$ is
 $(\beta, \gamma)$-recurrent
and $(\beta, \gamma,\varrho)$-recurrent.
We will use induction on $j$.
First assume that $j=2$. By construction
$\lbrace a_i \rbrace_{i=0}^2$ is $(\beta, \gamma)$-recurrent.
By construction
and Lemma~\ref{lem:bgrecvsbgdrecS99}(i), the sequence
$\lbrace a_i \rbrace_{i=0}^2$ is $(\beta, \gamma,\varrho)$-recurrent.
Similarly $\lbrace b_i \rbrace_{i=0}^2$ is $(\beta, \gamma)$-recurrent and
$(\beta, \gamma,\varrho)$-recurrent.
We are done for $j=2$.
Next assume that $j\geq 3$.
By $E(1,j)$,
\begin{align}
(a_0+a_1-b_0-b_1)(a_{j-1}-b_j)=
(a_0-b_1)(a_{j-1}+a_j-b_{j-1}-b_j).
\label{eq:E1n}
\end{align}
By $E(2,j)$,
\begin{gather}
(a_0+a_1+a_2-b_0-b_1-b_2)(a_{j-2}-b_j)\nonumber\\
\qquad{} =(a_0-b_2)(a_{j-2}+a_{j-1}+a_j-b_{j-2}-b_{j-1}-b_j).\label{eq:E2n}
\end{gather}
The equations (\ref{eq:E1n}), (\ref{eq:E2n}) give a linear system in the unknowns $a_j$, $b_j$. For this system the coefficient matrix
is
\begin{gather*}
C =
	 \left(
	 \begin{matrix}
	 a_0-b_1 & a_1-b_0 \\
	 a_0-b_2 & a_1+a_2-b_0-b_1
		 \end{matrix}
		 \right).
\end{gather*}
We have
\begin{gather*}
\det (C) = (a_0-b_1)(a_1+a_2-b_0-b_1)-(a_0-b_2) (a_1-b_0).
\end{gather*}
In the above equation we simplify the right-hand side using~(\ref{eq:twin2}), to obtain
\begin{gather*}
\det (C) = (a_1-b_0)(a_1-b_1).
\end{gather*}
Therefore $\det (C) \not=0$, so the system~(\ref{eq:E1n}),~(\ref{eq:E2n})
has a unique solution for $a_j$, $b_j$.
We now describe the solution. By induction the sequences
$\lbrace a_i \rbrace_{i=0}^{j-1}$,
$\lbrace b_i \rbrace_{i=0}^{j-1}$ are
$(\beta,\gamma)$-recurrent and
$(\beta,\gamma,\varrho)$-recurrent.
Define $a'_j$, $b'_j$ such that
\begin{gather*}
a_{j-2}-\beta a_{j-1} + a'_j = \gamma,
\qquad
b_{j-2}-\beta b_{j-1} + b'_j = \gamma.
\end{gather*}
Consider the two sequences
\begin{gather}
a_0,a_1, \ldots, a_{j-1}, a'_j;\label{eq:aprime}\\
b_0,b_1, \ldots, b_{j-1}, b'_j.\label{eq:bprime}
\end{gather}
By construction, each of~(\ref{eq:aprime}),
(\ref{eq:bprime}) is $(\beta,\gamma)$-recurrent.
By construction and
Lemma~\ref{lem:bgrecvsbgdrecS99}(i),
each of~(\ref{eq:aprime}),
(\ref{eq:bprime}) is $(\beta,\gamma,\varrho)$-recurrent.
We show that $a_j=a'_j$ and
 $b_j=b'_j$.
The sequences~(\ref{eq:aprime}),
(\ref{eq:bprime}) are recurrent and twins, so by
Proposition~\ref{lem:betaGam} they satisfy~$E(1,j)$ and~$E(2,j)$.
 Therefore, the equations~(\ref{eq:E1n}),
(\ref{eq:E2n}) still hold if we replace~$a_j$,~$b_j$
by $a'_j$, $b'_j$. We mentioned earlier that
the system~(\ref{eq:E1n}), (\ref{eq:E2n}) has a unique solution for
$a_j$, $b_j$.
By these comments $a_j=a'_j$ and $b_j=b'_j$.
Consequently each of $\lbrace a_i \rbrace_{i=0}^{j}$,
$\lbrace b_i \rbrace_{i=0}^{j}$ is
$(\beta,\gamma)$-recurrent and
$(\beta,\gamma,\varrho)$-recurrent.
The above argument shows that each of
 $\lbrace a_i \rbrace_{i=0}^{n}$,
$\lbrace b_i \rbrace_{i=0}^{n}$ is
$(\beta,\gamma)$-recurrent and
$(\beta,\gamma,\varrho)$-recurrent.
\end{proof}

\begin{Lemma}\label{lem:ce}
Assume that $a_{i-1}=b_i$ for $1 \leq i \leq n$.
Then the equation $E(1,j)$ holds for $1 \leq j \leq n$.
However, in general it is not the case that $E(i,j)$ holds
for $0 \leq i\leq j\leq n$.
\end{Lemma}
\begin{proof} For $0\leq i \leq j\leq n$ the equation
$E(i,j)$ becomes
\begin{gather}
(a_i-b_0)(a_{j-i}-b_j) =
(a_0-b_i)(a_{j}-b_{j-i}).
\label{eq:eijCase}
\end{gather}
If $i=1$ then each side of~(\ref{eq:eijCase}) is zero,
so $E(1,j)$ holds.
Assume that $n=3$ and
\begin{gather*}
a_0=0=b_1, \qquad a_1=1=b_2, \qquad a_2=0=b_3,\qquad a_3=1, \qquad b_0=0.
\end{gather*}
Then~(\ref{eq:eijCase}) fails for $i=2$ and $j=3$.
\end{proof}

\section{The classification of the double lowering data}\label{section12}

Recall the data
$\lbrace a_i\rbrace_{i=0}^{N-1}$,
$\lbrace b_i\rbrace_{i=0}^{N-1}$ from~(\ref{eq:data}). In this section we obtain
necessary and sufficient conditions for this
data to be double lowering.
In view of Lemma~\ref{lem:n2}
 we assume $N\geq 3$.

\begin{Theorem}\label{thm:hr}
Let $N$ denote an integer at least~$3$, or $\infty$. Let
\begin{gather}
\lbrace a_i \rbrace_{i=0}^{N-1}, \qquad
\lbrace b_i \rbrace_{i=0}^{N-1} \label{eq:datamth}
\end{gather}
denote scalars in $\mathbb F$ such that
\begin{gather}
a_0 + a_1 + \cdots + a_{i-1} \not=
b_0 + b_1 + \cdots + b_{i-1}, \qquad 1 \leq i \leq N.
\label{eq:dontforget}
\end{gather}
Then the data~\eqref{eq:datamth} is double lowering if and only if at least one
of the following $(i)$--$(iv)$ holds:
\begin{enumerate}\itemsep=0pt
\item[$(i)$] $a_{i-1}= b_i$ for $1 \leq i \leq N-1$;
\item[$(ii)$] $a_{i}= b_{i-1}$ for $1 \leq i \leq N-1$;
\item[$(iii)$] there exists $\theta \in \mathbb F$ such that
\begin{gather*}
 a_0 \not=\theta, \qquad b_0 \not=\theta, \qquad
 a_i = \theta, \qquad b_i = \theta, \qquad 1 \leq i \leq N-2,\\
\frac{\theta- a_{N-1}}{\theta-b_0} =\frac{\theta- b_{N-1}}{\theta-a_0}
 \qquad
\text{if} \ N\not=\infty.
\end{gather*}
\item[$(iv)$] the sequences \eqref{eq:datamth} are recurrent and twins.
\end{enumerate}
\end{Theorem}
\begin{proof} First assume that at least one of (i)--(iii) holds.
Then~(\ref{eq:datamth}) is double lowering, by Lemmas~\ref{lem:five},~\ref{lem:fivedual},~\ref{lem:deg}.
Next assume that (iv) holds
and $N\not=\infty$.
By Proposition~\ref{lem:betaGam} (with $n=N-1$) the equations
$E(i,j)$ hold for $0 \leq i \leq j\leq N-1$.
We delete $a_{N-1}$, $b_{N-1}$ from~(\ref{eq:datamth}) and consider the data
\begin{gather}
\lbrace a_i \rbrace_{i=0}^{N-2}, \qquad
\lbrace b_i \rbrace_{i=0}^{N-2}. \label{eq:Rdatamth}
\end{gather}
By Lemma~\ref{lem:n2} and
induction on $N$, we may assume that
the sequences (\ref{eq:Rdatamth}) are double lowering.
By Proposition \ref{prop:feas3}(i),~(ii) (with $N$ replaced by
$N-1$) we find that~(\ref{eq:datamth}) is double lowering.
Next assume that
(iv) holds and $N=\infty$. Then for all integers $j\geq 2$
the sequences
$\lbrace a_i \rbrace_{i=0}^j$,
$\lbrace b_i \rbrace_{i=0}^j$ are recurrent and twins.
Consequently the data
$\lbrace a_i \rbrace_{i=0}^j$,
$\lbrace b_i \rbrace_{i=0}^j$ is double lowering.
Therefore the data
$\lbrace a_i \rbrace_{i=0}^\infty$,
$\lbrace b_i \rbrace_{i=0}^\infty$ is double lowering.
We are done in one direction.

We now reverse the direction.
Next assume that
(\ref{eq:datamth}) is double lowering. We break the argument into
cases.

{\bf Case} $a_0 = b_1$.
We show that (i) holds.
We have $\eta_i(a_0)\not=0$ for $0 \leq i \leq 1$.
By assumption the data~(\ref{eq:datamth}) is double lowering,
so the data
$\lbrace a_i \rbrace_{i=0}^j$,
$\lbrace b_i \rbrace_{i=0}^j$ is double lowering for
$1 \leq j\leq N-1$. Applying Proposition~\ref{prop:feas3}(i),~(ii) repeatedly
(with $N$ replaced by $2,3,\ldots, N-1$) we find that
$E(1,j)$ holds for
$2\leq j\leq N-1$. Using these equations
and (\ref{eq:dontforget})
we routinely obtain $a_{j-1} = b_j$ for $2 \leq j \leq N-1$.
This and $a_0=b_1$ implies (i).

{\bf Case} $a_1 = b_0$. Interchanging the roles of
$\lbrace a_i \rbrace_{i=0}^{N-1}$,
$\lbrace b_i \rbrace_{i=0}^{N-1}$ in the previous case, we find
that (ii) holds.

{\bf Case}
$a_0 \not = b_1$,
$a_1 \not = b_0$,
$a_1 = b_1$.
We show that (iii) holds.
Define $\theta=a_1=b_1$, and note that $a_0 \not=\theta$,
$b_0 \not=\theta$.
We have $\eta_i(a_0)\not=0$ for $0 \leq i \leq 2$.
By assumption the data~(\ref{eq:datamth}) is double lowering,
so the data
$\lbrace a_i \rbrace_{i=0}^j$,
$\lbrace b_i \rbrace_{i=0}^j$ is double lowering for
$1 \leq j\leq N-1$. Applying Proposition~\ref{prop:feas3}(i),~(ii) repeatedly
(with $N$ replaced by $2,3,\ldots, N-1$) we find that
$E(i,j)$ holds for $1 \leq i \leq 2$
and $i+1\leq j\leq N-1$.
Next we show that $a_k=b_k=\theta$ for $2 \leq k \leq N-2$.
We will use induction on $k$. Assume $N\geq 4$; otherwise
there is nothing to prove.
Using $E(1,2)$, $E(1,3)$
and
$a_0 \not=b_1$
we obtain
\begin{gather}
a_2 = b_2 \zeta + \theta (1-\zeta),\label{eq:a2}
\\
a_{3} = b_{3} \zeta + b_2 \big(1-\zeta^2\big) -\theta \zeta (1-\zeta),\label{eq:a3}
\end{gather}
where
\begin{gather}
\zeta=\frac{b_0-\theta}{a_0-\theta}.\label{eq:xi}
\end{gather}
In the equation $E(2,3)$ we eliminate $a_2$, $a_3$ using~(\ref{eq:a2}),
(\ref{eq:a3}). We evaluate the result using~(\ref{eq:dontforget}) (with $i=3$), to obtain
$b_2=\theta$.
In~(\ref{eq:a2}) we set $b_2=\theta$
to obtain $a_2=\theta$.
Next assume that $3 \leq k \leq N-2$.
By induction each of
 $a_2,a_3,\ldots, a_{k-1}$,
 $b_2,b_3,\ldots, b_{k-1}$ is equal to $\theta$.
Using this we evaluate $E(1,k)$, $E(1,k+1)$ to obtain
\begin{gather}
a_k = b_k \zeta + \theta (1-\zeta),\label{eq:ak}\\
a_{k+1} = b_{k+1} \zeta + b_k \big(1-\zeta^2\big) -\theta \zeta (1-\zeta).\label{eq:akp1}
\end{gather}
Using~(\ref{eq:ak}),~(\ref{eq:akp1}) we evaluate $E(2,k+1)$ to obtain
$b_k=\theta$. In~(\ref{eq:ak})
we set $b_k=\theta$
to obtain $a_k=\theta$.
We have shown that
$a_k = b_k = \theta$ for $2 \leq k \leq N-2$.
Now for $N\not=\infty$ we use $E(1,N-1)$
to obtain
\begin{gather*}
a_{N-1} = b_{N-1} \zeta + \theta (1-\zeta).
\end{gather*}
Evaluating this using~(\ref{eq:xi}) we obtain
\begin{gather*}
\frac{\theta- a_{N-1}}{\theta-b_0} =\frac{\theta- b_{N-1}}{\theta-a_0}.
\end{gather*}
We have shown that (iii) holds.

{\bf Case} $a_0 \not = b_1$,
$a_1 \not = b_0$,
$a_1 \not= b_1$.
We show that (iv) holds.
Using Proposition~\ref{prop:feas3}(i),~(ii)
as in the previous case, we find that
$E(i,j)$ holds for $1 \leq i \leq 2$
and $i+1\leq j\leq N-1$.
Now by Proposition~\ref{prop:hr}
(with $n=N-1$),
the sequences~(\ref{eq:datamth}) are recurrent and twins.
We have shown that~(iv) holds.
\end{proof}

\section[L and Delta for twin recurrent data]{$\boldsymbol{\mathcal L}$ and $\boldsymbol{\Delta}$ for twin recurrent data}\label{section13}

Our goal for the rest of the paper is to give a comprehensive
description of $\mathcal L$ and $\Delta$ for twin recurrent data.
We will focus on Case~I in Lemma~\ref{lem:twin}, or more precisely, an adjusted version
of this case as described in
Section~\ref{section3}.

For the rest of this paper we assume that $N$ is an
integer at least~2, or~$\infty$.
Recall the double lowering space~$\mathcal L$
for the data~(\ref{eq:data}).
For the rest of this paper, fix nonzero $a,b,q \in \mathbb F$ and assume
\begin{gather}
 a_i= a q^i + a^{-1} q^{-i},
 \qquad
b_i= b q^i + b^{-1} q^{-i}
\label{eq:aibi}
\end{gather}
for $ 0 \leq i \leq N-1$. By~(\ref{eq:nondeg}) we have
 $a \not=b$ and
\begin{gather*}
q^i \not=1, \qquad
abq^{i-1} \not=1, \qquad 1 \leq i \leq N.
\end{gather*}
\begin{Note}\label{note:qi}
The data
$\lbrace a_i \rbrace_{i=0}^{N-1}$,
$\lbrace b_i \rbrace_{i=0}^{N-1}$
is unchanged if we replace
\begin{gather*}
q \mapsto q^{-1},
\qquad
a\mapsto a^{-1},
\qquad
b\mapsto b^{-1}.
\end{gather*}
\end{Note}

\begin{Lemma}\label{lem:rtwin}
The sequences \eqref{eq:data} are recurrent and twins.
\end{Lemma}
\begin{proof}
By Lemmas~\ref{lem:closedformthreetermS99} and~\ref{lem:twin}.
\end{proof}
\begin{Corollary}
The data~\eqref{eq:data} is double lowering.
\end{Corollary}
\begin{proof}
By Lemma~\ref{lem:rtwin}
 along with
Lemma~\ref{lem:n2}
and
 Theorem~\ref{thm:hr}(iv).
\end{proof}

Our next general goal is to describe
the polynomials $\lbrace \tau_i\rbrace_{i=0}^{N-1}$,
$\lbrace \eta_i\rbrace_{i=0}^{N-1}$ from~(\ref{eq:tau}),~(\ref{eq:eta}),
the parameters $\lbrace \vartheta_i \rbrace_{i=0}^{N}$
from Definition~\ref{def:vartheta},
and the
map~$\Delta$ from~(\ref{def:Delta}).

We mention some formulas for later use.
\begin{Lemma}\label{lem:aidiff}
For $1 \leq i \leq N-1$,
\begin{alignat*}{3}
&qa_i - a_{i-1} = \big(q-q^{-1}\big)aq^{i},\qquad && q b_i - b_{i-1} = \big(q-q^{-1}\big)bq^{i},&\\
&a_i - a_{i-1} = (q-1)\big(aq^{i-1}-a^{-1}q^{-i}\big),\qquad && b_i - b_{i-1} = (q-1)\big(bq^{i-1}-b^{-1}q^{-i}\big),& \\
&a_i - q a_{i-1} = \big(1-q^2\big)a^{-1}q^{-i},\qquad && b_i - q b_{i-1} = \big(1-q^2\big)b^{-1}q^{-i}.
\end{alignat*}
\end{Lemma}
\begin{proof} Use~(\ref{eq:aibi}).
\end{proof}

\begin{Lemma}\label{lem:preL}
For $0 \leq i \leq N$,
\begin{gather*}
a_0 + a_1 + \cdots + a_{i-1} = \frac{1-q^i}{1-q} \big(a + a^{-1} q^{1-i}\big),\\
b_0 + b_1 + \cdots + b_{i-1} = \frac{1-q^i}{1-q} \big(b + b^{-1} q^{1-i}\big).
\end{gather*}
\end{Lemma}
\begin{proof} Use~(\ref{eq:aibi}).
\end{proof}

Next we describe $\lbrace \vartheta_i \rbrace_{i=0}^{N}$.
We give two versions.
\begin{Lemma}
\label{lem:vartheta}
For $0 \leq i \leq N$,
\begin{enumerate}\itemsep=0pt
\item[$(i)$]
$\displaystyle{
\vartheta_i = \frac{1-q^i}{1-q} \frac{1-abq^{i-1}}{1-ab} q^{1-i}}$;
\item[$(ii)$]
$\displaystyle{
 \vartheta_i = \frac{1-q^{-i}}{1-q^{-1}} \frac{1-a^{-1}b^{-1}q^{1-i}}{1-a^{-1}b^{-1}}
 q^{i-1}}$.
\end{enumerate}
\end{Lemma}
\begin{proof} (i) By Definition~\ref{def:vartheta}
and Lemma \ref{lem:preL}.
(ii)~By Note~\ref{note:qi} and~(i) above.
\end{proof}

We mention some formulas for later use.
\begin{Lemma}\label{lem:vthdiff}
For $0 \leq i \leq N-1$,
\begin{gather*}
q \vartheta_{i+1} - \vartheta_i = \frac{q+ab-(q+1)ab q^i}{1-ab},\\
 \vartheta_{i+1} - \vartheta_i = \frac{q^{-i}-abq^i}{1-ab},\\
 \vartheta_{i+1} - q \vartheta_i = \frac{(1+q)q^{-i}-q-ab}{1-ab}.
\end{gather*}
\end{Lemma}
\begin{proof} Use
Lemma~\ref{lem:vartheta}.
\end{proof}

We recall some notation. For an element $\alpha$ in
any algebra, define
\begin{gather*}
(\alpha;q)_i = (1-\alpha) (1-\alpha q) \cdots \big(1-\alpha q^{i-1}\big),
\qquad i \in \mathbb N.
\end{gather*}
We interpret $(\alpha;q)_0=1$.
We remark that for $j\geq i\geq 0$,
\begin{gather}
\big(q^{-j};q\big)_i(q;q)_{j-i} =(-1)^i(q;q)_jq^{\binom{i}{2}}q^{-ij}.\label{eq:qint}
\end{gather}

\begin{Lemma}\label{lem:usefulvt}
For $0 \leq i \leq N$,
\begin{gather}
\vartheta_1 \vartheta_2 \cdots \vartheta_i =
\frac{(q;q)_i (ab;q)_i q^{-\binom{i}{2}}}{(1-q)^i(1-ab)^i}.
\label{eq:vthprod}
\end{gather}
For $0 \leq i\leq j\leq N$,
\begin{gather}
\vartheta_j \vartheta_{j-1} \cdots \vartheta_{j-i+1} =
\frac{
\big(q^{-j};q\big)_i \big(a^{-1}b^{-1}q^{1-j};q\big)_i q^{i(j-i)} q^{\binom{i}{2}}}
{\big(1-q^{-1}\big)^i\big(1-a^{-1}b^{-1}\big)^i}.\label{eq:vthprod2}
\end{gather}
\end{Lemma}
\begin{proof}
To obtain~(\ref{eq:vthprod}),
use Lemma \ref{lem:vartheta}(i).
To obtain (\ref{eq:vthprod2}), use Lemma~\ref{lem:vartheta}(ii).
\end{proof}

\begin{Lemma}\label{lem:ijN}
For $0 \leq i\leq j\leq N$,
\begin{gather*}
\left[\begin{matrix} j \\
i \end{matrix} \right]_\vartheta =
\frac{
\big(q^{-j};q\big)_i
\big(a^{-1} b^{-1} q^{1-j};q\big)_i
q^{ij}
a^ib^i}
{
(q;q)_i
(ab;q)_i
}.
\end{gather*}
\end{Lemma}
\begin{proof} Evaluate (\ref{eq:brackij}) using Lemma~\ref{lem:usefulvt}.
\end{proof}

We comment on the notation.
Let $y$ denote an indeterminate. Let
${\mathbb F}\big\lbrack y, y^{-1}\big\rbrack$
denote the algebra consisting of the Laurent polynomials in $y$ that have
all coefficients
in $\mathbb F$.
This algebra has an automorphism that sends $y\mapsto y^{-1}$.
An element of
${\mathbb F}\big\lbrack y, y^{-1}\big\rbrack$
that is fixed by the automorphism is called symmetric.
The symmetric elements form a subalgebra of
${\mathbb F}\big\lbrack y, y^{-1}\big\rbrack$
called its symmetric part.
There exists an injective algebra homomorphism
$\iota\colon {\mathbb F}\lbrack x \rbrack \to
{\mathbb F}\big\lbrack y, y^{-1}\big\rbrack$
that sends $x \mapsto y+y^{-1}$.
The image of
${\mathbb F}\lbrack x \rbrack$
under $\iota$ is the symmetric part of
${\mathbb F}\lbrack y, y^{-1}\rbrack$.
Via $\iota$ we
identify
${\mathbb F}\lbrack x \rbrack$
 with the symmetric part of
${\mathbb F}\lbrack y, y^{-1}\rbrack$.

\begin{Lemma}\label{lem:qte}For $0 \leq i \leq N$,
\begin{enumerate}\itemsep=0pt
\item[$(i)$] $\tau_i = (-1)^i a^{-i} q^{-\binom{i}{2}}(ay;q)_i \big(ay^{-1};q\big)_i$;
\item[$(ii)$] $\eta_i = (-1)^i b^{-i} q^{-\binom{i}{2}} (by;q)_i \big(by^{-1};q\big)_i$.
\end{enumerate}
In the above lines $x = y+ y^{-1}$.
\end{Lemma}
\begin{proof}(i) We use~(\ref{eq:tau}) and~(\ref{eq:aibi}).
For $0 \leq j \leq i-1$,
\begin{gather*}
x - a_{j} = y+y^{-1} - a q^{j} - a^{-1} q^{-j}= -a^{-1} q^{-j} \big(1-ayq^{j}\big)\big(1-ay^{-1} q^{j}\big).
\end{gather*}
The result follows. (ii) Similar to the proof of~(i) above.
\end{proof}

\begin{Lemma}\label{lem:tauai}
For $0 \leq i \leq N$,
\begin{enumerate}\itemsep=0pt
\item[$(i)$] $\tau_i(b_0) = (-1)^i a^{-i} q^{-\binom{i}{2}}
(ab;q)_i \big(ab^{-1};q\big)_i$;
\item[$(ii)$] $\eta_i(a_0) = (-1)^i b^{-i} q^{-\binom{i}{2}}
(ab;q)_i \big(a^{-1}b;q\big)_i$.
\end{enumerate}
\end{Lemma}
\begin{proof} (i) Set $y=b$ in
Lemma~\ref{lem:qte}(i), and use $b_0=b+b^{-1}$.
(ii)~Similar to the proof of (i) above.
\end{proof}

Our data is double lowering, so $\mathcal L\not=0$.
For the rest of the paper, let $ \psi \in \mathcal L$ be normalized.
The maps $\Delta$, $\psi$ are related
by~(\ref{eq:DSum}),
(\ref{eq:DiSum}). Our next goal is to interpret
(\ref{eq:DSum}),
(\ref{eq:DiSum}) using the
 $q$-exponential function.
This function is defined as follows.
For locally nilpotent $T \in \operatorname{End}(V)$,
\begin{gather}
\exp_q(T) = \sum_{i=0}^N \frac{q^{\binom{i}{2}}(1-q)^i T^i}{ (q;q)_i}.\label{eq:qExp}
\end{gather}
The map
$\exp_q(T)$ is invertible; its inverse is
\begin{gather}\label{eq:qExpi}
\exp_{q^{-1}}(-T) = \sum_{i=0}^N\frac{(-1)^i (1-q)^i T^i}{ (q;q)_i}.
\end{gather}

\begin{Lemma}\label{lem:qexp}
For locally nilpotent $T \in \operatorname{End}(V)$,
\begin{gather*}
\bigl(1-(q-1)T\bigr) \exp_q(qT) = \exp_q (T).
\end{gather*}
\end{Lemma}
\begin{proof} To verify this equation, for $0 \leq i \leq N$ compare the
coefficient of $T^i$ on each side.
\end{proof}

\begin{Proposition}\label{prop:DeltaExp}
We have
\begin{gather}
\Delta = \exp_q \big( a^{-1} \xi \psi \big)
 \exp_{q^{-1}} \big({-}b^{-1} \xi \psi \big),\label{eq:DeltaProd}
\end{gather}
where $\xi = 1-ab$.
\end{Proposition}
\begin{proof}For $0 \leq j \leq N$
we compare the coefficient of $\psi^j$ on each side of~(\ref{eq:DeltaProd}). For the left-hand side
these coefficients are obtained from~(\ref{eq:DSum}).
We require
\begin{gather}\label{eq:Require}
\frac{\eta_j(a_0)}{\vartheta_1 \vartheta_2 \cdots \vartheta_j}
=
\sum_{i=0}^j \frac{q^{\binom{i}{2}}a^{-i}(1-q)^i \xi^i }{(q;q)_i}
\frac{(-1)^{j-i} b^{i-j}(1-q)^{j-i} \xi^{j-i} }{(q;q)_{j-i}}.
\end{gather}
By (\ref{eq:vthprod}) and the construction,
\begin{gather*}
\vartheta_1 \vartheta_2 \cdots \vartheta_j = (q;q)_j (ab;q)_j
q^{-\binom{j}{2}} (1-q)^{-j} \xi^{-j}.
\end{gather*}
By Lemma~\ref{lem:tauai}(ii),
\begin{gather*}
\eta_j(a_0)= (-1)^j b^{-j} q^{-\binom{j}{2}}(ab;q)_j \big(a^{-1}b;q\big)_j.
\end{gather*}
Using these comments and (\ref{eq:qint}),
the equation (\ref{eq:Require}) becomes
\begin{gather}
\big(zq^{-j};q\big)_j = \sum_{i=0}^j \frac{\big(q^{-j};q\big)_i z^i}{(q;q)_i},
\label{eq:Require2}
\end{gather}
where $z= a^{-1}bq^j$.
Basic hypergeometric series are discussed in~\cite{gr,koekoek}. In~(\ref{eq:Require2}) the sum on the right is the basic hypergeometric series
\begin{gather*}
 {}_1\phi_0 \left(
 \genfrac{}{}{0pt}{}
 {q^{-j} }
 {- }
 \,\bigg\vert \, q; z \right).
\end{gather*}
This observation reveals
that (\ref{eq:Require2}) is an instance of the $q$-binomial theorem
\cite[Section~1.3]{gr}.
The result follows.
\end{proof}

Proposition~\ref{prop:DeltaExp} gives a factorization of~$\Delta$.
We now investigate the factors.

\begin{Lemma}\label{lem:factors}
For $0 \leq i \leq N$,
\begin{gather}
\exp_{q^{-1}} \big({-}a^{-1} \xi \psi \big) \eta_i=
\exp_{q^{-1}} \big({-}b^{-1} \xi \psi \big) \tau_i,\label{eq:wi}
\\
\exp_{q} \big(b^{-1} \xi \psi \big) \eta_i
=
\exp_{q} \big(a^{-1} \xi \psi \big) \tau_i,
\label{eq:wip}
\end{gather}
where $\xi = 1-ab$.
\end{Lemma}
\begin{proof} By
(\ref{eq:DeltaProd}) and the comments above Lemma
\ref{lem:qexp},
\begin{gather}
\exp_{q^{-1}} \big({-}a^{-1} \xi \psi \big) \Delta =
\exp_{q^{-1}} \big({-}b^{-1} \xi \psi \big).
\label{eq:2one}
\end{gather}
To obtain~(\ref{eq:wi}),
apply each side of
(\ref{eq:2one}) to $\tau_i$ and evaluate the result using~(\ref{def:Delta}).
For the equation
(\ref{eq:DeltaProd}), the two factors on the right commute; swapping
these factors and proceeding as above,
\begin{gather}
\exp_{q} \big(b^{-1} \xi \psi \big) \Delta =\exp_{q} \big(a^{-1} \xi \psi \big).\label{eq:2oneAlt}
\end{gather}
To obtain (\ref{eq:wip}), apply each side of~(\ref{eq:2oneAlt}) to $\tau_i$ and evaluate the result using~(\ref{def:Delta}).
\end{proof}

\begin{Definition}\label{def:wi}
For $0 \leq i \leq N$ let $w_i$ (resp.~$w'_i$) denote
the common value of~(\ref{eq:wi}) (resp.~(\ref{eq:wip})).
For notational convenience define $w_{-1}=0$ and $w'_{-1}=0$.
\end{Definition}

\begin{Lemma}\label{lem:four}
For $0 \leq i \leq N$,
\begin{alignat*}{3}
& \tau_i =\exp_{q} \big(b^{-1} \xi \psi \big) w_i,
\qquad &&
w_i =
\exp_{q^{-1}} \big({-}b^{-1} \xi \psi \big) \tau_i,&\\
& \eta_i =
\exp_{q} \big(a^{-1} \xi \psi \big) w_i,
\qquad &&
w_i =
\exp_{q^{-1}} \big({-}a^{-1} \xi \psi \big) \eta_i&
\end{alignat*}
and
\begin{alignat*}{3}
& \tau_i =
\exp_{q^{-1}} \big({-}a^{-1} \xi \psi \big) w'_i,
\qquad &&
w'_i
=
\exp_{q} \big(a^{-1} \xi \psi \big) \tau_i,&\\
& \eta_i=
\exp_{q^{-1}} \big({-}b^{-1} \xi \psi \big) w'_i
\qquad &&
w'_i =
\exp_{q} \big(b^{-1} \xi \psi \big) \eta_i.&
\end{alignat*}
In the above lines $\xi = 1-ab$.
\end{Lemma}

\begin{proof} By Definition~\ref{def:wi} and the comments above Lemma~\ref{lem:qexp}.
\end{proof}

\begin{Note}\label{note:wprime}
Referring to Definition \ref{def:wi}, the polynomials
$\lbrace w'_i\rbrace_{i=0}^N$ are obtained
from the polynomials $\lbrace w_i\rbrace_{i=0}^N$ by replacing
$q \mapsto q^{-1}$,
$a\mapsto a^{-1}$,
$b\mapsto b^{-1}$.
\end{Note}

\begin{Example}
\label{eq:wsmallex}
The following {\rm (i)--(iii)} hold:
\begin{enumerate}\itemsep=0pt
\item[(i)]
$w_0 =1$;
\item[(ii)]
$w_1$ is equal to each of
\begin{gather*}
\tau_1-(1-ab)b^{-1} \tau_0,
\qquad
\eta_1-(1-ab)a^{-1} \eta_0,
\qquad
x - a^{-1} - b^{-1};
\end{gather*}
\item[(iii)] $w_2$ is equal to each of
\begin{gather*}
 \tau_2 -\big(q^{-1}+1\big) (1-abq)b^{-1} \tau_1
+ q^{-1}(1-ab)(1-abq) b^{-2} \tau_0,
\\
 \eta_2 -\big(q^{-1}+1\big) (1-abq)a^{-1} \eta_1
+ q^{-1}(1-ab)(1-abq) a^{-2} \eta_0,
\\
 \big(x-a^{-1}-b^{-1}\big)\big(x-q^{-1}a^{-1}-q^{-1}b^{-1}\big)+
\big(q^{-1}-1\big)\big(1-a^{-1}b^{-1}\big).
\end{gather*}
\end{enumerate}
To get $w'_0$, $w'_1$, $w'_2$ replace
$q \mapsto q^{-1}$,
$a\mapsto a^{-1}$,
$b\mapsto b^{-1}$ in (i)--(iii) above.
\end{Example}

\begin{Lemma}\label{lem:wibasis}
The following $(i)$--$(iii)$ hold:
\begin{enumerate}\itemsep=0pt
\item[$(i)$]
for $0 \leq i \leq N$
the polynomials $w_i$, $w'_i$ are monic with degree
$i$;
\item[$(ii)$] for $0 \leq n\leq N$, each of
$\lbrace w_i \rbrace_{i=0}^n$,
$\lbrace w'_i \rbrace_{i=0}^n$ is a basis for the vector space $V_n$;
\item[$(iii)$]
each of $\lbrace w_i \rbrace_{i=0}^N$,
 $\lbrace w'_i \rbrace_{i=0}^N$ is a basis for the vector space $V$.
\end{enumerate}
\end{Lemma}
\begin{proof} (i)
By~(\ref{eq:Normal})
and Definition~\ref{def:wi}.
(ii),~(iii)~By~(i) above.
\end{proof}

\begin{Lemma}\label{lem:wjhyper} For $0 \leq j \leq N$,
\begin{align*}
w'_j &= a^{-j} (ab;q)_j \sum_{i=0}^j
\frac{
\big(q^{-j};q\big)_i
(ay;q)_i
\big(ay^{-1};q\big)_i q^i}
{
(ab;q)_i
(q;q)_i
}
\\
 &= a^{-j} (ab;q)_j\,
 {}_3\phi_2 \left(
 \genfrac{}{}{0pt}{}
 {q^{-j}, ay, ay^{-1} }
 {ab, 0}
 \,\bigg\vert \, q; q \right),
\end{align*}
where $x=y+y^{-1}$.
To get $w_j$ from $w'_j$, replace
$q \mapsto q^{-1}$,
$a\mapsto a^{-1}$,
$b\mapsto b^{-1}$.
\end{Lemma}
\begin{proof}In the equation
$w'_j = \exp_q\big(a^{-1} \xi \psi \big) \tau_j$,
expand the $q$-exponential using~(\ref{eq:qExp}), and evaluate the result using
the equation on the left in~(\ref{eq:Normal}). This yields
$w'_j = \sum_{i=0}^j \alpha_i \tau_i$ where for $0 \leq i \leq j$,
\begin{gather*}
\alpha_i = \frac{q^{\binom{j-i}{2}} (1-q)^{j-i} a^{i-j} \xi^{j-i}
\vartheta_j \vartheta_{j-1} \cdots \vartheta_{i+1}
}{(q;q)_{j-i}}.
\end{gather*}
Evaluating this using
\begin{gather*}
\vartheta_j \vartheta_{j-1} \cdots \vartheta_{i+1} =
\frac{
\vartheta_1 \vartheta_{2} \cdots \vartheta_{j}}
{
\vartheta_1 \vartheta_{2} \cdots \vartheta_{i}}
\end{gather*}
and~(\ref{eq:qint}),
(\ref{eq:vthprod}) we obtain
\begin{gather*}
\alpha_i = \frac{(-1)^i a^{i-j} (ab;q)_j \big(q^{-j};q\big)_i q^{\binom{i}{2}} q^i} {(ab;q)_i (q;q)_i}.
\end{gather*}
The polynomial $\tau_i$ is given in Lemma~\ref{lem:qte}(i). The result follows.
\end{proof}

\begin{Note}\label{note:3trec}
The polynomials $\lbrace w_i \rbrace_{i=0}^N$ and
$\lbrace w'_i \rbrace_{i=0}^N$ are in the
Al-Salam/Chihara family
\cite[Section~14.8]{koekoek}
if $N=\infty$,
 and the
dual $q$-Krawtchouk family \cite[Section~14.17]{koekoek}
if $N\not=\infty$. The Al-Salam/Chihara and dual $q$-Krawtchouk
polynomials satisfy a 3-term recurrence; the details will be
given in Lemmas~\ref{prop:3trec} and \ref{prop:3trecprme} below.
\end{Note}

Going forward we focus on $\lbrace w_i \rbrace_{i=0}^N$;
similar results hold for $\lbrace w'_i \rbrace_{i=0}^N$.

\begin{Lemma}\label{lem:psiW}
We have
\begin{gather*}
\psi w_i = \vartheta_i w_{i-1}, \qquad 0 \leq i \leq N.
\end{gather*}
\end{Lemma}
\begin{proof} By Definition~\ref{def:wi} and since $\psi \tau_i = \vartheta_i \tau_{i-1}$.
\end{proof}

Our next general goal is to describe in more detail
how the bases
$\lbrace \tau_i\rbrace_{i=0}^N$,
$\lbrace \eta_i\rbrace_{i=0}^N$,
$\lbrace w_i \rbrace_{i=0}^N$ are related.
To this end, we introduce some maps
$K, B, M \in \operatorname{End}(V)$.

\begin{Definition}\label{def:kbm}
\rm Define $K, B, M\in \operatorname{End}(V)$ such that
for $0 \leq i \leq N$,
\begin{gather*}
K \tau_i = q^{-i} \tau_i,
\qquad
B \eta_i = q^{-i} \eta_i,
\qquad
M w_i = q^{-i} w_i.
\end{gather*}
\end{Definition}
Each of $K$, $B$, $M$ is invertible.

\begin{Lemma}\label{lem:KBM}
The following $(i)$--$(iii)$ hold:
\begin{enumerate}\itemsep=0pt
\item[$(i)$]
$K\psi = q \psi K$;
\item[$(ii)$]
$B\psi = q \psi B$;
\item[$(iii)$]
$M\psi = q \psi M$.
\end{enumerate}
\end{Lemma}
\begin{proof} (i)~The vectors
$\lbrace \tau_i\rbrace_{i=0}^N$ form a basis for $V$.
For $0 \leq i \leq N$,
\begin{gather*}
K \psi \tau_i =\vartheta_i K \tau_{i-1} = \vartheta_i q^{1-i} \tau_{i-1},
 \qquad
q \psi K \tau_i = q^{1-i} \psi \tau_i = \vartheta_i q^{1-i} \tau_{i-1}.
\end{gather*}
Therefore $K\psi = q \psi K$.
(ii),~(iii)~Similar to the proof of~(i) above.
\end{proof}

\begin{Lemma}\label{lem:BDDK}
The following $(i)$--$(iii)$ hold:
\begin{enumerate}\itemsep=0pt
\item[$(i)$]
$B \Delta = \Delta K$;
\item[$(ii)$]
$K \exp_q \big(b^{-1}\xi \psi\big) =
 \exp_q \big(b^{-1}\xi \psi\big) M$;
\item[$(iii)$]
$B\exp_q \big(a^{-1}\xi \psi\big) =
 \exp_q \big(a^{-1}\xi \psi\big) M$.
\end{enumerate}
\end{Lemma}
\begin{proof}(i) For $0 \leq i \leq N$,
apply each side to $\tau_i$ and use~(\ref{def:Delta}) along with Definition~\ref{def:kbm}.

(ii), (iii) For $0 \leq i \leq N$,
apply each side to $w_i$
and use Lemma~\ref{lem:four} along with Definition~\ref{def:kbm}.
\end{proof}

\begin{Proposition}\label{prop:MKB}
The following $(i)$--$(iv)$ hold:
\begin{enumerate}\itemsep=0pt
\item[$(i)$] $KM^{-1} = I + (q-1)\big(a-b^{-1}\big) \psi$;
\item[$(ii)$] $M^{-1}K = I + (q^{-1}-1)\big(b^{-1}-a\big) \psi$;
\item[$(iii)$] $BM^{-1} = I + (q-1)\big(b-a^{-1}\big) \psi$;
\item[$(iv)$] $M^{-1}B = I + \big(q^{-1}-1\big)\big(a^{-1}-b\big) \psi$.
\end{enumerate}
\end{Proposition}
\begin{proof} (i)
 The map $T= b^{-1} \xi \psi$ is locally nilpotent.
We have
 $K T K^{-1} = qT$ by
 Lemma~\ref{lem:KBM}(i),
and $K \exp_q(T) = \exp_q(T) M$ by
Lemma~\ref{lem:BDDK}(ii).
By these comments and
 Lemma~\ref{lem:qexp},
\begin{align*}
\exp_q( T)&= \big(I - (q-1)T\big) \exp_q(qT)
= \big(I - (q-1)T\big) \exp_q\big(KTK^{-1}\big)
\\
&= \big(I - (q-1)T\big) K \exp_q(T)K^{-1}
= \bigl(I - (q-1)T\bigr) \exp_q(T)M K^{-1}
\\
&= \exp_q(T) \big(I - (q-1)T\big) M K^{-1}.
\end{align*}
By this and since $\exp_q( T)$ is invertible,
\begin{gather*}
I= \bigl(I - (q-1)T\bigr) M K^{-1}.
\end{gather*}
The result follows from this and $\xi = 1-ab$.

(ii) By (i) above and $ M \psi = q \psi M$.
(iii),~(iv)~Similar to the proof of~(i),~(ii) above.
\end{proof}

\begin{Corollary} We have
\begin{gather}
\psi = \frac{bq}{1-ab} \frac{M^{-1}K-KM^{-1}}{(q-1)^2},\label{eq:psiv1}\\
\psi = \frac{aq}{1-ab} \frac{M^{-1}B-BM^{-1}}{(q-1)^2}.\label{eq:psiv2}
\end{gather}
\end{Corollary}
\begin{proof} To get~(\ref{eq:psiv1}) use
Proposition~\ref{prop:MKB}(i),~(ii).

To get~(\ref{eq:psiv2}) use Proposition~\ref{prop:MKB}(iii),~(iv).
\end{proof}

\begin{Proposition}\label{prop:KBgivesM}
We have
\begin{gather*}
M = \frac{bK-aB}{b-a}.
\end{gather*}
\end{Proposition}
\begin{proof}
Compute $b$ times
Proposition~\ref{prop:MKB}(i)
minus $a$ times
Proposition~\ref{prop:MKB}(iii).
\end{proof}

\begin{Lemma}\label{lem:denom}
Each of the following is invertible:
\begin{gather*}
aI - bB^{-1}K, \qquad
bI - aK^{-1}B, \qquad
a^{-1}I - b^{-1}BK^{-1}, \qquad
b^{-1}I - a^{-1}KB^{-1}.
\end{gather*}
\end{Lemma}
\begin{proof} By
Proposition~\ref{prop:KBgivesM}
and since~$M$ is invertible.
\end{proof}

Our next goal is to describe how $K$, $B$ are related.
We will use the following result.

\begin{Lemma}\label{prop:pM}
We have
\begin{gather}
\psi M = \frac{1}{1-q} \frac{ab}{1-ab} \frac{K-B}{a-b},\label{eq:pM}\\
M \psi = \frac{q}{1-q} \frac{ab}{1-ab} \frac{K-B}{a-b}.\label{eq:Mp}
\end{gather}
\end{Lemma}
\begin{proof} To get~(\ref{eq:pM}), subtract
Proposition~\ref{prop:MKB}(iii) from
Proposition~\ref{prop:MKB}(i).
To get~(\ref{eq:Mp}) from~(\ref{eq:pM}), use $M\psi = q \psi M$.
\end{proof}

\begin{Proposition}\label{thm:BKBK}
We have
\begin{gather*}
(bK-aB)(K-B) = q(K-B)(bK-aB).
\end{gather*}
\end{Proposition}
\begin{proof} We have $M\psi = q \psi M$ so
$M (\psi M) = q (\psi M) M$.
Evaluate this using Proposition~\ref{prop:KBgivesM} and~(\ref{eq:pM}).
\end{proof}

We mention some reformulations of Proposition~\ref{thm:BKBK}.

\begin{Corollary}\label{lem:altrel}
We have
\begin{gather*}
0 = aB^2-\frac{bq-a}{q-1} BK - \frac{aq-b}{q-1}KB + bK^2
\end{gather*}
and
\begin{gather*}
\big(bI-aK^{-1}B\big)\big(I-KB^{-1}\big) = q \big(I-K^{-1}B\big)\big(aI-bKB^{-1}\big),
\\
\big(aI-bB^{-1}K\big)\big(I-KB^{-1}\big) = q \big(I-B^{-1}K\big)\big(aI-bKB^{-1}\big),
\\
\big(aI-bB^{-1}K\big)\big(I-BK^{-1}\big) = q \big(I-B^{-1}K\big)\big(bI-aBK^{-1}\big),
\\
\big(bI-aK^{-1}B\big)\big(I-BK^{-1}\big) = q \big(I-K^{-1}B\big)\big(bI-aBK^{-1}\big).
\end{gather*}
\end{Corollary}

\begin{Proposition}\label{prop:KBi}
We have
\begin{gather}
KB^{-1} = \frac{I + (q-1)\big(a-b^{-1}\big)\psi}{I + (q-1)\big(b-a^{-1}\big)\psi},
\label{eq:KB1}
\\
BK^{-1} = \frac
{I + (q-1)\big(b-a^{-1}\big)\psi}
{I + (q-1)\big(a-b^{-1}\big)\psi},
\label{eq:KB2}
\\
K^{-1}B = \frac{I + \big(q^{-1}-1\big)\big(a^{-1}-b\big)\psi}{I + \big(q^{-1}-1\big)\big(b^{-1}-a\big)\psi},
\label{eq:KB3}
\\
B^{-1}K = \frac
{I + \big(q^{-1}-1\big)\big(b^{-1}-a\big)\psi}
{I + \big(q^{-1}-1\big)\big(a^{-1}-b\big)\psi}.
\label{eq:KB4}
\end{gather}
In the above fractions the denominator is invertible since
$\psi$ is locally nilpotent.
\end{Proposition}
\begin{proof} To get~(\ref{eq:KB1}),
equate the two expressions for $M^{-1}$ obtained from
Proposition~\ref{prop:MKB}(i) and~(iii).
To get~(\ref{eq:KB2}) from~(\ref{eq:KB1}), compute the inverse of each side.
To get~(\ref{eq:KB3}),
equate the two expressions for~$M^{-1}$
obtained from Proposition~\ref{prop:MKB}(ii) and~(iv).
To get~(\ref{eq:KB4}) from~(\ref{eq:KB3}), compute the inverse of
each side.
\end{proof}

\begin{Lemma}\label{lem:BKcom}
The following mutually commute:
\begin{gather*}
\psi,\qquad
KB^{-1}, \qquad
BK^{-1}, \qquad
K^{-1}B, \qquad
B^{-1}K.
\end{gather*}
\end{Lemma}
\begin{proof} By Proposition~\ref{prop:KBi}.
\end{proof}

\begin{Proposition}\label{prop:psiBK}
We have
\begin{gather*}
\psi= \frac{1}{q-1} \frac{1}{1-ab}
\frac{I-KB^{-1}}
{b^{-1}I - a^{-1} KB^{-1}},
\\
\psi = \frac{1}{q-1}
\frac{1}{1-ab}
\frac{I-BK^{-1}}
{a^{-1}I - b^{-1} BK^{-1}},
\\
\psi = \frac{1}{q^{-1}-1}
\frac{1}{1-a^{-1}b^{-1}}
\frac{I-K^{-1}B}
{bI - a K^{-1}B},
\\
\psi = \frac{1}{q^{-1}-1}
\frac{1}{1-a^{-1}b^{-1}}
\frac{I-B^{-1}K}
{aI - b B^{-1}K}.
\end{gather*}
In the above fractions the denominator is invertible by Lemma~{\rm \ref{lem:denom}}.
\end{Proposition}
\begin{proof} In Proposition~\ref{prop:KBi} solve for~$\psi$.
\end{proof}

\begin{Proposition}\label{thm:qWeyl}
We have
\begin{gather}
\frac{qM^{-1}K-KM^{-1}}{q-1} = I,\label{eq:wq1}\\
\frac{qM^{-1}B-BM^{-1}}{q-1} = I.\label{eq:wq2}
\end{gather}
\end{Proposition}
\begin{proof} To get~(\ref{eq:wq1}) use Proposition~\ref{prop:MKB}(i) and~(ii).
To get~(\ref{eq:wq2}) use Proposition~\ref{prop:MKB}(iii) and~(iv).
\end{proof}

We have a comment.
\begin{Lemma}The relations in
Proposition~{\rm \ref{thm:BKBK}}
and Corollary~{\rm \ref{lem:altrel}}
still hold if we replace
\begin{gather*}
q\mapsto q^{-1}, \qquad
a\mapsto a^{-1}, \qquad
b\mapsto b^{-1}, \qquad
K\mapsto K^{-1}, \qquad
B\mapsto B^{-1}.
\end{gather*}
\end{Lemma}
\begin{proof} Use Proposition~\ref{thm:BKBK} and Lemma~\ref{lem:BKcom}.
\end{proof}

We recall some notation.
Let ${\rm Mat}_{N+1}(\mathbb F)$ denote the set of
 $N+1$ by $N+1$ matrices that have all entries in~$\mathbb F$.
We index the rows and columns by $0,1,\ldots, N$.
Let $\lbrace v_i \rbrace_{i=0}^N$ denote a basis for~$V$.
For $M \in
{\rm Mat}_{N+1}(\mathbb F)$ and
$T \in \operatorname{End}(V)$, we say that {\it $M$ represents $T$ with
respect to
$\lbrace v_i \rbrace_{i=0}^N$} whenever
$Tv_j = \sum_{i=0}^N M_{i,j} v_i$ for $0 \leq j \leq N$.

Our next goal is to display the matrices that represent
$\psi$, $K^{\pm 1}$, $M^{\pm 1}$, $B^{\pm 1}$ with respect to
the bases $\lbrace \tau_i\rbrace_{i=0}^N$,
$\lbrace w_i\rbrace_{i=0}^N$,
 $\lbrace \eta_i\rbrace_{i=0}^N$ of $V$.

\begin{Definition}\label{def:psimat}
Let $\widehat \psi$ denote the matrix in
${\rm Mat}_{N+1}(\mathbb F)$
that has $(i-1,i)$-entry $\vartheta_i$
for $1 \leq i \leq N$,
and all other entries 0. Thus
\begin{gather*}
 {\widehat \psi} =
	 \left(
	 \begin{matrix}
	 0 & \vartheta_1 & & & & {\bf 0} \\
	 & 0 & \vartheta_2 & & & \\
		 & & 0 & \cdot & &
		 \\
		 && & \cdot & \cdot&
		 \\
		 & & & & \cdot & \vartheta_N \\
		 {\bf 0} & & & & & 0
		 \end{matrix}
		 \right).
\end{gather*}
\end{Definition}

\begin{Lemma}\label{Psirep}
The matrix $\widehat \psi$
represents $\psi$ with respect to
$\lbrace \tau_i \rbrace_{i=0}^N$ and
$\lbrace w_i \rbrace_{i=0}^N$ and
$\lbrace \eta_i \rbrace_{i=0}^N$.
\end{Lemma}
\begin{proof} By~(\ref{eq:Normal})
and Lemma~\ref{lem:psiW}.
\end{proof}

\begin{Lemma}\label{Diag}
The matrix ${\rm diag}\big(1,q^{-1},q^{-2},\ldots, q^{-N}\big)$
represents $K$ $($resp.~$M)$ $($resp.~$B)$ with respect to
$\lbrace \tau_i \rbrace_{i=0}^N$
$($resp.~$\lbrace w_i \rbrace_{i=0}^N)$
$($resp.~$\lbrace \eta_i \rbrace_{i=0}^N)$.
\end{Lemma}
\begin{proof} By Definition~\ref{def:kbm}.
\end{proof}

\begin{Lemma}\label{Kwj}
We give the matrix in ${\rm Mat}_{N+1}(\mathbb F)$ that represents~$K$ with respect to
$\lbrace w_i \rbrace_{i=0}^N$. The $(i,i)$-entry is $q^{-i}$
for $0 \leq i \leq N$. The $(i-1,i)$-entry is
$\big(1-q^{-i}\big)\big(a-b^{-1}q^{1-i}\big)$ for $1 \leq i \leq N$.
All other entries are~$0$.
\end{Lemma}
\begin{proof} Use $KM^{-1} = I+ (q-1)\big(a-b^{-1}\big)\psi$ and
Lemmas~\ref{lem:vartheta}(ii),~\ref{Psirep},~\ref{Diag}.
\end{proof}

\begin{Lemma}\label{Bwj}
We give the matrix in
${\rm Mat}_{N+1}(\mathbb F)$ that represents $B$ with respect to
$\lbrace w_i \rbrace_{i=0}^N$. The $(i,i)$-entry is $q^{-i}$
for $0 \leq i \leq N$. The $(i-1,i)$-entry is
$\big(1-q^{-i}\big)\big(b-a^{-1}q^{1-i}\big)$ for $1 \leq i \leq N$.
All other entries are~$0$.
\end{Lemma}
\begin{proof} Use $BM^{-1} = I+ (q-1)\big(b-a^{-1}\big)\psi$ and
Lemmas~\ref{lem:vartheta}(ii),~\ref{Psirep},~\ref{Diag}.
\end{proof}

\begin{Lemma}\label{Mitauj}
We give the matrix in ${\rm Mat}_{N+1}(\mathbb F)$ that represents $M^{-1}$ with respect to
$\lbrace \tau_i \rbrace_{i=0}^N$. The $(i,i)$-entry is~$q^{i}$
for $0 \leq i \leq N$. The $(i-1,i)$-entry is
$\big(q^i-1\big)\big(aq^{i-1}-b^{-1}\big)$ for $1 \leq i \leq N$.
All other entries are~$0$.
\end{Lemma}
\begin{proof} Use $M^{-1}K = I+ \big(q^{-1}-1\big)\big(b^{-1}-a\big)\psi$ and
Lemmas~\ref{lem:vartheta}(i), \ref{Psirep},~\ref{Diag}.
\end{proof}

\begin{Lemma}\label{Mietaj}
We give the matrix in ${\rm Mat}_{N+1}(\mathbb F)$ that represents $M^{-1}$ with respect to
$\lbrace \eta_i \rbrace_{i=0}^N$. The $(i,i)$-entry is $q^{i}$
for $0 \leq i \leq N$. The $(i-1,i)$-entry is
$\big(q^i-1\big)\big(bq^{i-1}-a^{-1}\big)$ for $1 \leq i \leq N$.
All other entries are~$0$.
\end{Lemma}
\begin{proof} Use $M^{-1}B = I+ \big(q^{-1}-1\big)\big(a^{-1}-b\big)\psi$ and
Lemmas~\ref{lem:vartheta}(i), \ref{Psirep},~\ref{Diag}.
\end{proof}

We give a variation on Proposition~\ref{prop:MKB}.

\begin{Lemma} We have
\begin{gather}
MK^{-1} = \sum_{i=0}^N (-1)^i (1-q)^i (1-ab)^i b^{-i} \psi^i,
\label{eq:MKinv}
\\
K^{-1}M = \sum_{i=0}^N (-1)^i (1-q)^i q^{-i} (1-ab)^i b^{-i} \psi^i,
\label{eq:KinvM}
\\
MB^{-1} = \sum_{i=0}^N (-1)^i (1-q)^i (1-ab)^i a^{-i} \psi^i,
\label{eq:MBinv}
\\
B^{-1}M = \sum_{i=0}^N (-1)^i (1-q)^i q^{-i} (1-ab)^i a^{-i} \psi^i.
\label{eq:BinvM}
\end{gather}
\end{Lemma}
\begin{proof} For each equation in Proposition~\ref{prop:MKB},
take the inverse of each side and evaluate the result using
Lemma~\ref{lem:gs}.
\end{proof}

\begin{Lemma}\label{Mtauj}
We give the matrix in
${\rm Mat}_{N+1}(\mathbb F)$ that represents $M$ with respect to
$\lbrace \tau_i \rbrace_{i=0}^N$.
The $(i,j)$-entry is
\begin{gather*}
\frac{(-1)^{j-i} b^{i-j} (ab;q)_j (q;q)_j q^{\binom{i}{2}-\binom{j}{2}-j}
}{
(ab;q)_i (q;q)_i}
\end{gather*}
for $0 \leq i\leq j\leq N$.
All other entries are~$0$.
\end{Lemma}
\begin{proof} Use~(\ref{eq:vthprod}), (\ref{eq:MKinv})
and Lemmas~\ref{Psirep},~\ref{Diag}.
\end{proof}

\begin{Lemma}\label{Kiwj}
We give the matrix in ${\rm Mat}_{N+1}(\mathbb F)$ that represents $K^{-1}$ with respect to
$\lbrace w_i \rbrace_{i=0}^N$.
The $(i,j)$-entry is
\begin{gather*}
\frac{(-1)^{j-i} b^{i-j} (ab;q)_j (q;q)_j q^{\binom{i}{2}-\binom{j}{2}+i}
}{
(ab;q)_i (q;q)_i}
\end{gather*}
for $0 \leq i\leq j\leq N$. All other entries are~$0$.
\end{Lemma}
\begin{proof} Use (\ref{eq:vthprod}), (\ref{eq:KinvM})
and Lemmas~\ref{Psirep},~\ref{Diag}.
\end{proof}

\begin{Lemma}\label{Metaj}
We give the matrix in
${\rm Mat}_{N+1}(\mathbb F)$ that represents $M$ with respect to
$\lbrace \eta_i \rbrace_{i=0}^N$.
The $(i,j)$-entry is
\begin{gather*}
\frac{(-1)^{j-i} a^{i-j} (ab;q)_j (q;q)_j q^{\binom{i}{2}-\binom{j}{2}-j}
}{
(ab;q)_i (q;q)_i}
\end{gather*}
for $0 \leq i\leq j\leq N$.
All other entries are~$0$.
\end{Lemma}
\begin{proof} Use (\ref{eq:vthprod}), (\ref{eq:MBinv}) and Lemmas~\ref{Psirep},~\ref{Diag}.
\end{proof}

\begin{Lemma}\label{Biwj}
We give the matrix in
${\rm Mat}_{N+1}(\mathbb F)$ that represents $B^{-1}$ with respect to
$\lbrace w_i \rbrace_{i=0}^N$.
The $(i,j)$-entry is
\begin{gather*}
\frac{(-1)^{j-i} a^{i-j} (ab;q)_j (q;q)_j q^{\binom{i}{2}-\binom{j}{2}+i}
}{(ab;q)_i (q;q)_i}
\end{gather*}
for $0 \leq i\leq j\leq N$.
All other entries are~$0$.
\end{Lemma}
\begin{proof} Use
(\ref{eq:vthprod}),
(\ref{eq:BinvM})
and Lemmas~\ref{Psirep},~\ref{Diag}.
\end{proof}

We give a variation on Proposition~\ref{prop:KBi}.

\begin{Lemma} We have
\begin{gather}
KB^{-1} = I + \frac{b-a}{b} \sum_{i=1}^N
(-1)^i (1-q)^i (1-ab)^i a^{-i} \psi^i,
\label{eq:KBinv}\\
BK^{-1} = I + \frac{a-b}{a} \sum_{i=1}^N
(-1)^i (1-q)^i (1-ab)^i b^{-i} \psi^i,
\label{eq:BKinv}\\
K^{-1}B = I + \frac{a-b}{a} \sum_{i=1}^N
(-1)^i (1-q)^i q^{-i} (1-ab)^i b^{-i} \psi^i,
\label{eq:KinvB}\\
B^{-1}K = I + \frac{b-a}{b} \sum_{i=1}^N
(-1)^i (1-q)^i q^{-i} (1-ab)^i a^{-i} \psi^i.
\label{eq:BinvK}
\end{gather}
\end{Lemma}
\begin{proof} Evaluate each equation in Proposition~\ref{prop:KBi} using Lemma~\ref{lem:gs}.
\end{proof}

\begin{Lemma}\label{Ketaj}
We give the matrix in
${\rm Mat}_{N+1}(\mathbb F)$ that represents $K$ with respect to
$\lbrace \eta_i \rbrace_{i=0}^N$.
The $(i,i)$-entry is $q^{-i}$ for $0 \leq i \leq N$.
The $(i,j)$-entry is
\begin{gather*}
\frac{b-a}{b}
\frac{(-1)^{j-i} a^{i-j} (ab;q)_j (q;q)_j q^{\binom{i}{2}-\binom{j}{2}-j}
}{(ab;q)_i (q;q)_i}
\end{gather*}
for $0 \leq i< j\leq N$.
All other entries are~$0$.
\end{Lemma}
\begin{proof} Use (\ref{eq:vthprod}), (\ref{eq:KBinv}) and Lemmas~\ref{Psirep},~\ref{Diag}.
\end{proof}

\begin{Lemma}\label{Btauj}
We give the matrix in ${\rm Mat}_{N+1}(\mathbb F)$ that represents~$B$ with respect to
$\lbrace \tau_i \rbrace_{i=0}^N$. The $(i,i)$-entry is $q^{-i}$ for $0 \leq i \leq N$.
The $(i,j)$-entry is
\begin{align*}
\frac{a-b}{a}
\frac{(-1)^{j-i} b^{i-j} (ab;q)_j (q;q)_j q^{\binom{i}{2}-\binom{j}{2}-j}
}{(ab;q)_i (q;q)_i}
\end{align*}
for $0 \leq i< j\leq N$. All other entries are~$0$.
\end{Lemma}
\begin{proof} Use (\ref{eq:vthprod}), (\ref{eq:BKinv}) and Lemmas~\ref{Psirep},~\ref{Diag}.
\end{proof}

\begin{Lemma}\label{Kietaj}
We give the matrix in
${\rm Mat}_{N+1}(\mathbb F)$ that represents $K^{-1}$ with respect to
$\lbrace \eta_i \rbrace_{i=0}^N$.
The $(i,i)$-entry is $q^{i}$ for $0 \leq i \leq N$.
The $(i,j)$-entry is
\begin{gather*}
\frac{a-b}{a}
\frac{(-1)^{j-i} b^{i-j} (ab;q)_j (q;q)_j q^{\binom{i}{2}-\binom{j}{2}+i}
}{
(ab;q)_i (q;q)_i}
\end{gather*}
for $0 \leq i< j\leq N$.
All other entries are~$0$.
\end{Lemma}
\begin{proof} Use (\ref{eq:vthprod}), (\ref{eq:KinvB}) and Lemmas~\ref{Psirep},~\ref{Diag}.
\end{proof}

\begin{Lemma}\label{Bitauj}
We give the matrix in
${\rm Mat}_{N+1}(\mathbb F)$ that represents $B^{-1}$ with respect to
$\lbrace \tau_i \rbrace_{i=0}^N$.
The $(i,i)$-entry is $q^{i}$ for $0 \leq i \leq N$.
The $(i,j)$-entry is
\begin{gather*}
\frac{b-a}{b}
\frac{(-1)^{j-i} a^{i-j} (ab;q)_j (q;q)_j q^{\binom{i}{2}-\binom{j}{2}+i}
}{(ab;q)_i (q;q)_i}
\end{gather*}
for $0 \leq i< j\leq N$.
All other entries are~$0$.
\end{Lemma}
\begin{proof} Use (\ref{eq:vthprod}),
(\ref{eq:BinvK}) and Lemmas~\ref{Psirep},~\ref{Diag}.
\end{proof}

We recall some notation. Let $\lbrace u_i\rbrace_{i=0}^N$
and $\lbrace v_i\rbrace_{i=0}^N$ denote bases for $V$.
By the {\it transition matrix from
 $\lbrace u_i\rbrace_{i=0}^N$
 to $\lbrace v_i\rbrace_{i=0}^N$} we mean the matrix
$T \in {\rm Mat}_{N+1}(\mathbb F)$ such that
$v_j = \sum_{i=0}^N T_{i,j} u_i$ for $0 \leq j \leq N$.

Our next goal is to display the transition matrices
between the bases
$\lbrace \tau_i\rbrace_{i=0}^N$,
$\lbrace w_i\rbrace_{i=0}^N$,
 $\lbrace \eta_i\rbrace_{i=0}^N$.
Recall the notation $\xi = 1-ab$.

Consider the following matrices:
\begin{gather}
\exp_q \big( a^{-1} \xi \widehat \psi \big),
\qquad
\exp_q \big( b^{-1} \xi \widehat \psi \big).
\label{eq:expa}
\end{gather}
Their inverses are
\begin{gather}
\exp_{q^{-1}} \big({-}a^{-1} \xi \widehat \psi \big),
\qquad
\exp_{q^{-1}} \big({-}b^{-1} \xi \widehat \psi \big).
\label{eq:expai}
\end{gather}
The matrices (\ref{eq:expa}), (\ref{eq:expai}) are upper triangular. We now give their
entries.

\begin{Lemma}\label{lem:expEntries}
For $0 \not= z \in \mathbb F$ the matrix
$\exp_q \big( z \xi \widehat \psi \big)$ is upper
triangular. Its $(i,j)$-entry is
\begin{gather*}
\frac{(-1)^i z^{j-i}
(ab;q)_j
(q^{-j};q)_i
q^{i+\binom{i}{2}}}
{(ab;q)_i (q;q)_i}
\end{gather*}
for $0 \leq i \leq j \leq N$.
The matrix
$\exp_{q^{-1}} \big( {-}z \xi \widehat \psi \big)$ is upper
triangular. Its $(i,j)$-entry is
\begin{gather*}
\frac{(-1)^j z^{j-i}
(ab;q)_j
(q^{-j};q)_i
q^{ij-\binom{j}{2}}}
{(ab;q)_i (q;q)_i}
\end{gather*}
for $0 \leq i \leq j \leq N$.
\end{Lemma}
\begin{proof}
Use
(\ref{eq:qint}),
(\ref{eq:vthprod}),
(\ref{eq:qExp}),
(\ref{eq:qExpi}) and Definition~\ref{def:psimat}.
\end{proof}

\begin{Lemma}\label{lem:4trans}
The transition matrices between the basis
$\lbrace w_i \rbrace_{i=0}^N$ and the bases
$\lbrace \tau_i \rbrace_{i=0}^N$,
$\lbrace \eta_i \rbrace_{i=0}^N$ are given in the table below:
\begin{center}
\begin{tabular}[t]{ccc}
 {\rm from} & {\rm to}
 &{\rm transition matrix}
\\
\hline
$\lbrace \tau_i\rbrace_{i=0}^N$ &
$\lbrace w_i\rbrace_{i=0}^N$
&$\exp_{q^{-1}} \big({-}b^{-1} \xi \widehat \psi \big)$
\\
$\lbrace w_i\rbrace_{i=0}^N$ &
$\lbrace \tau_i\rbrace_{i=0}^N$
&$\exp_{q} \big(b^{-1} \xi \widehat \psi \big)$
\\
$\lbrace \eta_i\rbrace_{i=0}^N$ &
$\lbrace w_i\rbrace_{i=0}^N$
&$\exp_{q^{-1}} \big(-a^{-1} \xi \widehat \psi \big)$
\\
$\lbrace w_i\rbrace_{i=0}^N$ &
$\lbrace \eta_i\rbrace_{i=0}^N$
&
$\exp_{q} \big(a^{-1} \xi \widehat \psi \big)$
\\
\end{tabular}
\end{center}
\end{Lemma}
\begin{proof} By Lemmas~\ref{lem:four},~\ref{Psirep}.
\end{proof}

Next we consider the product
\begin{gather}\label{eq:prodOne}
\exp_{q} \big(a^{-1} \xi \widehat \psi \big)
\exp_{q^{-1}} \big(-b^{-1} \xi \widehat \psi \big).
\end{gather}
The inverse of~(\ref{eq:prodOne}) is
\begin{gather}
\label{eq:prodInv}
\exp_{q} \bigl(b^{-1} \xi \widehat \psi \bigr)
\exp_{q^{-1}} \bigl(-a^{-1} \xi \widehat \psi \bigr).
\end{gather}
The matrices
(\ref{eq:prodOne}),
(\ref{eq:prodInv}) are upper triangular.
Shortly we will give their entries.

\begin{Lemma}\label{lem:doubletrans}
The matrix~\eqref{eq:prodOne} is the transition matrix from
the basis
$\lbrace \tau_i\rbrace_{i=0}^N$ to the basis
$\lbrace \eta_i\rbrace_{i=0}^N$.
The matrix~\eqref{eq:prodInv} is the transition matrix from
the basis
$\lbrace \eta_i\rbrace_{i=0}^N$ to the basis
$\lbrace \tau_i\rbrace_{i=0}^N$.
\end{Lemma}
\begin{proof}
By Lemma~\ref{lem:4trans}.
\end{proof}

\begin{Lemma} The matrix~\eqref{eq:prodOne} represents $\Delta$
with respect to
$\lbrace \tau_i \rbrace_{i=0}^N$ and
$\lbrace w_i \rbrace_{i=0}^N$ and
$\lbrace \eta_i \rbrace_{i=0}^N$.
The matrix~\eqref{eq:prodInv} represents $\Delta^{-1}$
with respect to
$\lbrace \tau_i \rbrace_{i=0}^N$ and
$\lbrace w_i \rbrace_{i=0}^N$ and
$\lbrace \eta_i \rbrace_{i=0}^N$.
\end{Lemma}
\begin{proof} By
Proposition~\ref{prop:DeltaExp}
and Lemma~\ref{Psirep}.
\end{proof}

\begin{Lemma} The matrix~\eqref{eq:prodOne} is upper triangular, with $(i,j)$-entry
\begin{gather}
 \eta_{j-i} (a_0)
\left[\begin{matrix} j \\
i \end{matrix} \right]_\vartheta
\label{eq:96entries}
\end{gather}
for $0 \leq i\leq j\leq N$.
The matrix~\eqref{eq:prodInv} is upper triangular, with $(i,j)$-entry
\begin{gather}\label{eq:97entries}
 \tau_{j-i} (b_0)
\left[\begin{matrix} j \\
i \end{matrix} \right]_\vartheta
\end{gather}
for $0 \leq i\leq j\leq N$.
To express~\eqref{eq:96entries},
\eqref{eq:97entries} in terms of
$a$, $b$, $q$ use
Lemmas~{\rm \ref{lem:ijN}}, {\rm \ref{lem:tauai}}.
\end{Lemma}
\begin{proof}
By Lemma~\ref{lem:doubletrans}, the matrix~(\ref{eq:prodOne}) is the transition matrix from
the basis
$\lbrace \tau_i\rbrace_{i=0}^N$ to the basis
$\lbrace \eta_i\rbrace_{i=0}^N$.
The entries of this matrix are obtained
from Proposition~\ref{prop:tfae4}(ii).
The entries of the matrix~(\ref{eq:prodInv}) are similarly obtained.
\end{proof}

Our next goal is to show how the map $A$
from Definition~\ref{def:A} is related to~$\psi$,
$K$,
$B$,
$M$.
\begin{Lemma}\label{lem:AvsKB}
On $V_{N-1}$,
\begin{gather}\label{eq:KA}
\frac{qKA-AK}{q-1} = a^{-1} K^2+ a I,\\
\label{eq:BA}
\frac{qBA-AB}{q-1} = b^{-1} B^2+ b I.
\end{gather}
\end{Lemma}
\begin{proof}
We first show (\ref{eq:KA}).
For $0 \leq i \leq N-1$ apply
each side of
(\ref{eq:KA}) to $\tau_i$, and evaluate the result using
Lemma
\ref{lem:Ataui} along with
(\ref{eq:aibi})
and
Definition
\ref{def:kbm}. We have
\begin{gather*}
KA\tau_i = K(a_i \tau_i + \tau_{i+1}) = q^{-i} a_i \tau_i + q^{-i-1}\tau_{i+1},\\
AK\tau_i = q^{-i}A\tau_i = q^{-i}(a_i \tau_i + \tau_{i+1}),\\
\big(a^{-1}K^2+ a I\big)\tau_i =\big(a^{-1}q^{-2i}+ a\big)\tau_i =
q^{-i}a_i \tau_i.
\end{gather*}
By these comments we obtain~(\ref{eq:KA}). Equation~(\ref{eq:BA}) is similarly obtained.
\end{proof}

\begin{Lemma}\label{lem:Avspsi}
On $V_{N-1}$,
\begin{gather}
q \psi A - A \psi = \frac{ (q+1)ab M^{-1} - (q+ab)I}{ab-1}.\label{eq:Avspsi}
\end{gather}
\end{Lemma}
\begin{proof}
For $0 \leq i \leq N-1$ apply each side of~(\ref{eq:Avspsi}) to $\tau_i$.
Using
Lemma~\ref{lem:Ataui} and~(\ref{eq:Normal}),
\begin{gather*}
\psi A \tau_i = \psi (a_i \tau_i + \tau_{i+1}) =a_i \vartheta_i \tau_{i-1} + \vartheta_{i+1} \tau_i,\\
A \psi \tau_i = \vartheta_i A \tau_{i-1} = \vartheta_i (a_{i-1} \tau_{i-1}+ \tau_i).
\end{gather*}
Using Lemma~\ref{lem:vartheta}(i)
and Lemma~\ref{Mitauj},
\begin{gather*}
M^{-1} \tau_i = q^i \tau_i + (q-1)\big(a-b^{-1}\big)q^{i-1}\vartheta_i\tau_{i-1}.
\end{gather*}
By the above comments and Lemmas~\ref{lem:aidiff},
\ref{lem:vthdiff} we get the result.
\end{proof}

\begin{Lemma}\label{lem:AvsMpsi}
On $V_{N-1}$,
\begin{gather}\label{eq:AvsMpsi}
\frac{ q A M^{-1} - M^{-1} A}{q-1} =
\big(a^{-1} + b^{-1}\big) I+
\big(q-q^{-1}\big)\big(1-a^{-1} b^{-1}\big)\psi.
\end{gather}
\end{Lemma}
\begin{proof} For $0 \leq i \leq N-1$ apply each side
of~(\ref{eq:AvsMpsi}) to~$\tau_i$. Evaluate the result
using~(\ref{eq:Normal}) and Lemmas~\ref{lem:Ataui},
\ref{Mitauj}
along with Lemmas~\ref{lem:aidiff}, \ref{lem:vthdiff}.
\end{proof}

\begin{Proposition}\label{prop:AvPsi}
On $V_{N-2}$,
\begin{gather*}
A^2 \psi - \big(q+q^{-1}\big)A \psi A + \psi A^2+ \big(q-q^{-1}\big)^2 \psi =
\frac{ (1-q)\big(1+q^{-1} ab\big)}{1-ab} A + \frac{\big(q-q^{-1}\big)(a+b)}{1-ab} I.
\end{gather*}
\end{Proposition}
\begin{proof} Let $X$ denote the expression on either side
of~(\ref{eq:Avspsi}).
Compute
\begin{gather*}
\frac{q A X - X A}{q-1}
\end{gather*}
and evaluate the result using Lemma~\ref{lem:AvsMpsi}.
\end{proof}

\begin{Proposition}\label{prop:ppA}
On $V_{N-1}$,
\begin{gather*}
\psi^2 A - \big(q+q^{-1}\big) \psi A \psi + A \psi^2 = \frac{(1-q)\big(1+ q^{-1}ab\big)}{1-ab}
\psi.
\end{gather*}
\end{Proposition}
\begin{proof} Let $X$ denote the expression on either side of~(\ref{eq:Avspsi}). Compute $qX \psi - \psi X$ and evaluate
the result using $q M^{-1}\psi = \psi M^{-1}$.
\end{proof}

\begin{Proposition}\label{prop:AAM}
On $V_{N-2}$,
\begin{gather*}
A^2 M^{-1} - \big(q+q^{-1}\big)AM^{-1}A+M^{-1}A^2 + \big(q-q^{-1}\big)^2 M^{-1}\\
\qquad{}=
(q-1)\big(q-q^{-1}\big)\big( q^{-1}+ a^{-1}b^{-1}\big) I - q^{-1}(q-1)^2 \big(a^{-1}+b^{-1}\big) A.
\end{gather*}
\end{Proposition}
\begin{proof} Let $Y$ denote the expression on either side of~(\ref{eq:AvsMpsi}). Compute $q YA - AY $
and evaluate the result using
Lemma~\ref{lem:Avspsi}.
\end{proof}

\begin{Proposition}\label{prop:MMA}On $V_{N-1}$,
\begin{gather*}
M^{-2}A - \big(q+q^{-1}\big) M^{-1} A M^{-1} + A M^{-2} =
(q-1)\big(q^{-1}-1\big)\big(a^{-1}+b^{-1}\big) M^{-1}.
\end{gather*}
\end{Proposition}
\begin{proof} Let $Y$ denote the expression on either side of~(\ref{eq:AvsMpsi}).
Compute $q M^{-1} Y - Y M^{-1}$ and evaluate the result using $q M^{-1} \psi = \psi M^{-1}$.
\end{proof}

In Note~\ref{note:3trec} we mentioned
a 3-term recurrence satisfied by the polynomials
$\lbrace w_i \rbrace_{i=0}^N$ and
$\lbrace w'_i \rbrace_{i=0}^N$.
Our next goal is to describe this recurrence.

\begin{Lemma}\label{prop:3trec} We have
\begin{gather*}
x w_i = w_{i+1} + q^{-i}\big(a^{-1}+b^{-1}\big)w_i + \big(1-q^{-i}\big)\big(1-q^{1-i}a^{-1}b^{-1}\big)w_{i-1}
\end{gather*}
for $0 \leq i \leq N-1$, where $w_0=1$ and $w_{-1}=0$.
\end{Lemma}
\begin{proof}The result holds for $i=0$, since $w_1=x-a^{-1}-b^{-1}$
by Example~\ref{eq:wsmallex}(ii).
Assume that $i\geq 1$. By Lemma~\ref{lem:wibasis}(i) there exist
scalars $\lbrace \alpha_k \rbrace_{k=0}^{i+1}$ in $\mathbb F$
such that
$xw_i = \sum_{k=0}^{i+1} \alpha_k w_k$ and $\alpha_{k+1}=1$.
To obtain $\lbrace \alpha_k \rbrace_{k=0}^{i}$, apply each side of~(\ref{eq:AvsMpsi}) to $w_i$ and evaluate the result using
Lemma~\ref{lem:psiW} along with
$M^{-1}w_j = q^j w_j$ for $0 \leq j \leq i$.
After a brief calculation this yields
$\alpha_i = q^{-i}\big(a^{-1}+b^{-1}\big)$ and
$\alpha_{i-1} =
\big(1-q^{-i}\big)\big(1-q^{1-i}a^{-1}b^{-1}\big)$ and
$\alpha_k = 0 $ for $0 \leq k \leq i-2$.
The result follows.
\end{proof}

\begin{Lemma}\label{prop:3trecprme} We have
\begin{gather*}
x w'_i = w'_{i+1} + q^{i}(a+b)w'_i + \big(1-q^{i}\big)\big(1-q^{i-1}ab\big)w'_{i-1}
\end{gather*}
$0 \leq i \leq N-1$, where $w'_0=1$ and $w'_{-1}=0$.
\end{Lemma}
\begin{proof} In Proposition~\ref{prop:3trec} replace
$q\mapsto q^{-1}$,
$a\mapsto a^{-1}$,
$b\mapsto b^{-1}$
and use Note~\ref{note:wprime}.
\end{proof}

This completes our description of $\cal L$ and $\Delta$ for the twin recurrent data from Case~I of Lemma~\ref{lem:twin}. In this description we encountered analogs of the results from Section~\ref{section1} about the double lowering operator of a tridiagonal pair. This connection suggests that double lowering operators on polynomials can be used to further develop the theory of
tridiagonal pairs; we hope to pursue this in the future.

\subsection*{Acknowledgement} The author would like to thank Kazumasa Nomura
for giving this paper a close reading and offering many valuable comments.

\pdfbookmark[1]{References}{ref}
\LastPageEnding

\end{document}